\documentclass[oneside,11pt]{amsart}

\newskip\stdskip                      % standard vertical space
\usepackage{pinlabel}  %%% the recommended graphics+labelling package
\usepackage{pictexwd,dcpic}

\usepackage{amscd,latexsym,amssymb}
\usepackage{bm}

  \usepackage{hyperref}  
  \hypersetup{%
  bookmarksnumbered=true,%
  bookmarks=true,%
  colorlinks=true,%
  linkcolor=blue,%
  citecolor=blue,%
  filecolor=blue,%
  menucolor=blue,%
  pagecolor=blue,%
  urlcolor=blue,%
  pdfnewwindow=true,%
  pdfstartview=FitBH}

\newtheorem{thm}{Theorem}[section]

\newtheorem{cor}[thm]{Corollary}
\newtheorem{lem}[thm]{Lemma}
\newtheorem{prop}[thm]{Proposition}

\theoremstyle{definition}
\newtheorem{defin}[thm]{Definition}

\theoremstyle{definition}

\newtheorem{exm}[thm]{Example}

\newtheorem{remark}[thm]{Remark}

\theoremstyle{remark}
\newtheorem*{rem}{Remark}
\newtheorem*{ack}{Acknowledgment}

\def\co{\colon\thinspace}
\def\wt{\widetilde}
\newcommand\numberthis{\addtocounter{equation}{1}\tag{\theequation}}

\begin{document}
\title{On products in a real moment-angle manifold}

\author{Li Cai}
\address{School of Mathematics and Systems Science, 
Chinese Academy of Sciences, Beijing 100190, China}
\email{l-cai@amss.ac.cn}
% Subject code(s)

% Key word(s)
\begin{abstract}    % type your abstract below
In this paper we give a necessary and sufficient condition for
a (real) moment-angle complex to be a topological manifold. 
The cup and cap products in a real moment-angle manifold are studied: 
the Poincar\'{e} duality via
cap products is equivalent to the Alexander duality of the defining 
complex $K$. Consequently, the cohomology ring (with coefficients integers) 
of a polyhedral product by pairs of disks and their bounding spheres 
is isomorphic to that of a differential graded algebra associated 
to $K$, and the dimensions of the disks.
\end{abstract}

\maketitle

%%%%%%%%%%%%%%%%%%%%   Start of main body of article
\section{Introduction}
Let $m$ be a positive integer, and let $K$ be an 
\emph{abstract simplicial complex} with vertex set 
$[m]=\{1,2,\ldots,m\}$. Thus $\emptyset\in K$, and any subset of 
$\sigma\subset[m]$ is a simplex of $K$, if $\sigma$ is. We use 
the notation $|K|$ for the geometric realization of $K$.
Denote by $(\underline{X},\underline{A})$ $m$ pairs of topological spaces 
$(X_i,A_i)$, $i=1,2,\ldots,m$, and let $\prod_{i=1}^m X_i$ 
be the $m$-fold Cartesian product of $X_i$. 
For $x=(x_i)_{i=1}^m\in \prod_{i=1}^m X_i$, define 
\begin{equation*}
  \sigma_x=\{i\in [m]\mid x_i\in X_i\setminus A_i\},
\end{equation*}
then the corresponding
\emph{polyhedral product} (following \cite{BBCG10a} and \cite{BP15}) is given by
\begin{equation}
  (\underline{X},\underline{A})^K=\{x\in \prod_{i=1}^m X_i\mid \sigma_x\in K\}.
  \label{def:pp}
\end{equation}
If all pairs $(X_i,A_i)$ are homeomorphic to a given
one, $(X,A)$, then $(\underline{X},\underline{A})^K$ shall be denoted 
by $(X,A)^K$.

$(D^1,S^0)^K$ is referred to as a \emph{real moment-angle complex} 
(cf.~\cite[Section 6.6]{BP02}). If it is a topological manifold, then 
we call it a \emph{real moment-angle manifold}.
Following \cite{BP02}, \cite{GL13} and \cite{BBCG12}, in this paper we 
focus on two problems:
\begin{enumerate}
  \item [(P-1)] the characterization of a real moment-angle manifold, and 
  \item [(P-2)] cup and cap products in its (co)homology.
\end{enumerate}

(P-1) was answered by M.~Davis, under the assumption that 
$K$ is a \emph{flag complex}, i.e., any set of vertices of $K$ that are pairwise
connected by edges spans a simplex of $K$ 
(cf.~\cite[Theorem 10.6.1, p.~197]{Dav08}).\footnote{The assumption implies 
  that $(D^1,S^0)^K$ is \emph{aspherical}, whose fundamental group is isomorphic
  to the commutator subgroup of the associated right-angled Coxeter group; 
  see \cite[pp.~11--12]{Dav08}.} 
It follows that 
$(D^1,S^0)^K$ is a topological $n$--manifold if and only if 
$|K|$ is a \emph{generalized homology} $(n-1)$--sphere (i.e., 
a homology $(n-1)$--manifold
having the homology of an $(n-1)$--sphere), with $\pi_1(|K|)=1$ when 
$n\not=1,2$.

In Section \ref{sec:Davis} we shall prove that Davis's 
characterization theorem still holds, without assuming the flagness of $K$ 
(see Theorem \ref{thm:Davis}). Together with the 
construction in \cite{BBCG10b} given by Bahri, Bendersky, Cohen and Gitler
(see Definition \ref{def:BBCG}),\footnote{This construction 
was intensely used in \cite{LdM89} and \cite{GL13} in the language
of simple polytopes.}
it follows that a \emph{moment-angle complex} 
$(D^2,S^1)^K$ is a topological ($n+m$)--manifold if and only if $|K|$
is a generalized homology $(n-1)$--sphere. 

Now we turn to (P-2). Unless otherwise stated, we always suppose 
that all coefficients taken in (co)homology groups are integers.

The additive structure of $H^*( (D^1,S^0)^K)$ is 
well-known.\footnote{See for example, 
 \cite{Dav83}, \cite{LdM89} from reflection 
groups, \cite{BBCG10a} from homotopy theory, also the 
Goresky--MacPherson Formula \cite{GM88} together with 
the fact that $(D^1,S^0)^K$ is the 
deformation retract of the \emph{coordinate subspace arrangement 
complement} 
$(\mathbb{R},\mathbb{R}\setminus\{0\})^K$ 
(see \cite[Theorem 8.9]{BP02}).}
It turns out that, we have isomorphisms
\begin{equation}
  H^p( (D^1,S^0)^K)\cong\bigoplus_{\omega\subset [m]}H^{p-1}(K_{\omega})
  \label{iso:coho}
\end{equation}
in all dimensions $p\geq 0$, where $K_{\omega}$ is the \emph{full subcomplex} 
with respect to $\omega$.

Let $J=(j_i)_{i=1}^m$ be an $m$-tuple of positive integers, 
and denote by $(\underline{D},\underline{S})^K$ the polyhedral
product with respect to pairs $(D^{j_i},S^{j_i-1})$ of $j_i$-disks
and their bounding spheres, $i=1,2,\ldots,m$.
With a general approach from homotopy theory, 
the information of the 
ring $H^*( (D^1,S^0)^K)$ is given in \cite{BBCG12}, as well as 
its relation with other polyhedral products, including $(\underline{D},\underline{S})^K$. 

Our approach here follows that of 
Baskakov, Buchstaber and Panov \cite{BBP04} and Panov \cite{Pan08} on 
$H^*( (D^2,S^1)^K)$:
we will show that 
$H^*( (D^1,S^0)^K)$ is also isomorphic to the cohomology a 
differential graded algebra $R^*_K$, which
is \emph{not} commutative in any sense (see \eqref{eq:relation1}, 
\eqref{eq:relation2}). 
Technically, we use the result of Whitney \cite{Whi38} on
the cup and cap products in a Cartesian product of simplicial 
complexs, and prove that it applies to the situation here 
(see Theorem \ref{thm:Whitney}).
Together with the method from \cite{BBCG12}, 
we will show that the ring
$H^*( (\underline{D},\underline{S})^K)$
can be understood uniformly (see Theorem \ref{thm:iso}, Remark \ref{rem:iso}).

In the language of the intersections of 
submanifolds, rules for the cup products were understood by 
Gitler and L\'{o}pez de Medrano \cite{GL13}. 
We shall follow their approach to make a comparison of 
the two rings $H^*((D^1,S^0)^K)$ 
and $H^*((D^2,S^1)^K)$ in Example \ref{exm:GL13} 
(compare \cite{GPW01}).

The paper is organized as follows.
Section \ref{sec:Davis} is devoted to the characterization
of a real moment-angle manifold; as a corollary, the 
characterization of a moment-angle manifold is given in 
Subsection \ref{subs:app1}. The cup and cap products in 
a real moment-angle complex, 
based on Whitney's formulae \eqref{eq:cup} and \eqref{eq:cap}, 
is given in Section \ref{sec:Whitney};  
while the proof the main theorem, Theorem \ref{thm:Whitney}, 
is postponed to Section \ref{proof:Whitney}. 
Most of the explicit
calculations are taken in Section \ref{app:Whitney}, based on
a special cellular (co)chain complex given in Subsection 
\ref{ss:rmac}.\footnote{The cellular (co)chain complex here 
  is used by Choi and Park \cite{CP13} 
  to show that, any odd torsion can appear 
  in the cohomology of a 
  \emph{real toric manifold} or a \emph{small cover}.}
In Section \ref{sec:coho}, we describe explicitly 
the cup products in the polyhedral product
$(\underline{D},\underline{S})^K$, via the multiplication in 
the corresponding differential graded algebra (see Theorem \ref{thm:iso}). 
Section \ref{proof:Whitney} is devoted to 
the proof of Theorem \ref{thm:Whitney}.

\begin{ack}
 I was led to the characterization of (real) moment-angle manifolds 
after discussions with Taras Panov and Hiroaki Ishida. 
Theorem \ref{thm:iso} follows the suggestions of Tony Bahri and Samuel Gitler. 
Finally, I would like to thank Osamu Saeki for the advice, guidance 
and many helpful comments. 

The author was supported in part by JSPS KAKENHI Grant Number 
23244008.
\end{ack}

\section{Davis's characterization theorem}\label{sec:Davis}
In this section, all manifolds are assumed to have 
no boundaries.

The dimension of a simplex $\sigma\in K$ is given by  
$\mathrm{card}(\sigma)-1$. Thus $\mathrm{dim}(\sigma)=-1$ 
if and only if $\sigma=\emptyset$.
Let $K'$ be the \emph{derived complex} of $K$, where 
a $k$-simplex is a chain $(\sigma_0,\sigma_1,\ldots,\sigma_k)$ 
of simplices in $K$, each of \emph{non-negative} dimension,
such that $\sigma_i\subset\sigma_{i+1}$ is a proper face, 
$i=0,1,\ldots, k-1$. 
Clearly $|K'|$ is the barycentric subdivision of $|K|$. 
Let $K'_+$ be the \emph{augmentation}
of $K'$: the dimension of the starting simplex in each 
chain of $K'_+$ can be negative. Denote by $|K'_+|$ the cone over $|K'|$, with the 
collapsed end point corresponding to $(\emptyset)$. (While 
$\emptyset\in K$ has no geometric meanings in $|K|$.) 
In what follows, we treat $|K'|$ as a subcomplex of $|K'_+|$.

Let $X_K$ be the intersection of $(D^1,S^0)^K$ with the 
first orthant of $\mathbb{R}^m$, namely the set $\{(x_i)_{i=1}^m\in\mathbb{R}^m
\mid x_i\geq 0, i=1,2,\ldots,m\}$. 
By \eqref{def:pp}, $X_K$ can be decomposed as the union of cubes 
\begin{equation}
  \bigcup_{\sigma\in K}C_{\sigma}; \quad C_{\sigma}=\prod_{i=1}^mY_i, \ 
  Y_i=\begin{cases}
	[0,1] & \text{if $i\in \sigma$,}\\
	\{1\} & \text{otherwise.}
  \end{cases}
  \label{def:cube}
\end{equation}

  \begin{lem}[{cf.~\cite[Chapter 4, pp.~53--55]{BP02}}] \label{lem:tri}
  Let $\varphi\co K\to X_K$ be the map sending each simplex $\sigma$ to 
  the point $\{x_i\}_{i=1}^m$, where $x_i=0$ if $i\in \sigma$ and 
  $x_i=1$ if otherwise. Then the mapping $\varphi'\co |K'_+|\to X_K$  
  which sends each $k$-simplex $|(\sigma_0,\sigma_1,\ldots,\sigma_k)|$ to the
  linear simplex spanned by $\varphi(\sigma_0),\varphi(\sigma_1),\ldots,\varphi(\sigma_k)$,
  yields a triangulation of $X_K$.
\end{lem}
\begin{proof}
  Suppose that $\sigma\in K$ is a $k$-simplex. 
  Let $K'_{\leq \sigma}\subset K'_+$ be the subcomplex 
  with maximal chains starting from $\emptyset$ and ending 
  with $\sigma$, i.e., each is of the form
  \[ \sigma'=(\sigma_{-1},\sigma_0,\ldots,\sigma_k)\]
  with $\mathrm{dim}(\sigma_i)=i$, such that $\sigma_k=\sigma$.
  For $i=-1,0,\ldots, k$, suppose $\varphi(\sigma_{i})=(x_{i,j})_{j=1}^m$. 
  By definition, $x_{i.j}$ has value of $0$ or $1$,  
  $x_{i,j}\geq x_{i+1,j}$ and two adjacent points  
  $\varphi(\sigma_i)$ and $\varphi(\sigma_{i+1})$ 
  have Euclidean distance $1$. 
  We see that $\varphi'(|\sigma'|)$ is spanned by lattice points 
  on a path connecting the two end-points on the diagonal of a cube.
  Then the standard way of triangulating a product
  of simplices implies that (cf.~Eilenberg and Steenrod \cite[p.~68]{ES52}),
  the restriction $\varphi'|_{|K'_{\leq \sigma}|}$ triangulates $C_{\sigma}$. 
  In this way $X_K$ is triangulated by $\varphi'$.
  \end{proof}

  Let $\mathbb{R}^m$ be endowed with the $\mathbb{Z}_2^m$-action 
  generated by $\{s_i\}_{i=1}^m$, with $s_i$ the reflection changing
  the sign of the $i$-th coordinate of each point. Denote by
  $X_i$ the subspace of $X_K$ fixed by $s_i$, $i=1,2,\ldots,m$.
  It can be checked that the inverse image of $X_i$ under 
  $\varphi'$ in the lemma above is the star of 
  $|(i)|$ in $|K'|$. Suppose that $Y_i=(\varphi')^{-1}(X_i)$. 
  Clearly $|K'|=\bigcup_{i=1}^mY_i$, and for $x\in |K'_+|$, 
  denote by $I_x$ the set $\{i\in [m]\mid x\in Y_i\}$ 
  (which can be empty).
  Let $\mathcal{U}(\mathbb{Z}_2^m,|K'_+|)$ be the \emph{basic
  construction} with respect to $\mathbb{Z}_2^m$, $|K'_+|$ 
  and $\{Y_i\}_{i=1}^m$ (cf.~\cite[Chapter 5]{Dav08}), 
  which is a $\mathbb{Z}_2^m$-space
  given by
  \begin{equation}
	\mathcal{U}(\mathbb{Z}_2^m,|K'_+|)=\left(\mathbb{Z}_2^m\times |K'_+|\right)/\sim,
	\label{def:U}
  \end{equation}
  where $(g,x)\sim (g',x')$ if and only if $x=x'\in |K'|$ and 
  $g^{-1}g'\in \langle s_i\rangle_{i\in I_x}$ ($g=g'$ if $I_x=\emptyset$),
  and the action follows $g'[g,x]=[g'g,x]$.
  It can be checked directly that the map
    \begin{equation}
   \begindc{\commdiag}[15]
   \obj(0,1)[aa]{$u_{\varphi'}\co \mathcal{U}(\mathbb{Z}_2^m,|K'_+|)$}
   \obj(7,1)[bb]{$ (D^1,S^0)^K$}
   \obj(0,0)[cc]{$[g,x]$}
   \obj(7,0)[dd]{$g \varphi'(x)$,}
   \mor{aa}{bb}{}
   \mor{cc}{dd}{}[+1,6]
 \enddc\label{def:uphi}
 \end{equation}
 is a homeomorphism preserving the $\mathbb{Z}_2^m$-actions
 on both sides. Then it follows from Lemma \ref{lem:tri} that 
 $u_{\varphi'}$ triangulates $(D^1,S^0)^K$, such that $|K'|$
 appears as the link of $\varphi'(|(\emptyset)|)=(1,1,\ldots,1)$.

 In what follows, a \emph{polyhedron} is a subset of Euclidean space, 
 in which each point has a neighborhood being a compact and 
 linear cone. It is well-known that a polyhedron can be triangulated.
 \begin{defin}\label{def:sphere}
A triangulated polyhedron $X$ is a \emph{homology $n$--manifold}
if the link of each $p$-simplex, $0\leq p\leq n$, 
has the homology of an $(n-1-p)$--sphere.\footnote{This is equivalent to 
  the definition using local homology groups, see 
  \cite[Exercise 64.1, p.~377]{Mun84}.} A homeomorphism $f$
  between two polyhedra is \emph{piecewise linear} 
  (abbreviated $\mathrm{PL}$), if 
  there exist suitable triangulations on both sides such that 
  $f$ is simplicial.

A polyhedral homology $n$--manifold 
$X$ is called a \emph{generalized homology $n$--sphere} (resp.~a $\mathrm{PL}$ $n$--sphere) 
if it has the homology of an $n$--sphere (resp.~is $\mathrm{PL}$ homeomorphic to 
the boundary of an $(n+1)$-simplex). 
$X$ (triangulated) is a \emph{piecewise linear $n$--manifold} if the 
 link of each vertex is a $\mathrm{PL}$ $(n-1)$--sphere. 
 \end{defin}

Let $(D^1,S^0)^K$ be equipped with the triangulation $u_{\varphi'}$ 
(see \eqref{def:uphi}).
We see that the necessary condition for $(D^1,S^0)^K$ to be a homology $n$--manifold 
(resp.~a $\mathrm{PL}$ $n$--manifold) is that $|K'|$ is a generalized homology
$(n-1)$--sphere (resp.~a $\mathrm{PL}$ $(n-1)$--sphere), where we can replace
$|K'|$ by $|K|$ since they are $\mathrm{PL}$ homeomorphic to each other.
In fact these conditions are also sufficient:
\begin{thm}\label{thm:Davis}
 The real moment-angle complex $(D^1,S^0)^K$ is a homology
$n$--manifold (resp.~a $\mathrm{PL}$ $n$--manifold), if and only if 
$|K|$ is a generalized homology $(n-1)$--sphere (resp.~a 
$\mathrm{PL}$ $(n-1)$--sphere); 
suppose that $(D^1,S^0)^K$ is a homology $n$--manifold,
then it is a topological manifold if and only if $|K|$ is simply 
connected when $n\geq 3$.
\end{thm}
\begin{remark}\label{rem:smooth}
  The argument of 
  Panov and Ustinovsky \cite[Theorem 2.2]{PU12} 
  can be used directly to prove that
  $(D^1,S^0)^K$ is homeomorphic to a smooth manifold, 
  if $K$ is induced from a \emph{complete simplicial fan} 
  (see also Tambour \cite{Tam12}),  
  including the well-known case that $|K|$ bounds a convex polytope.  
\end{remark}
The famous Edwards--Freedman Theorem (see \cite[Theorem 10.4.10, p. 194]{Dav08} 
and references therein) 
asserts that, a triangulated polyhedral homology $n$--manifold ($n\geq 3$) 
is a topological manifold,
if and only if the link of each vertex is simply connected. 
Therefore, it suffices to check the link of each vertex.

We proceed with the proposition below, whose proof will be given after
that of Theorem \ref{thm:Davis}.
Recall that for two disjoint polyhedra $X$ and $Y$ embedded in $\mathbb{R}^N$, 
their \emph{exterior join exists}, if any two line segments, each joining two points
in $X$ and $Y$ respectively, meet in at most one common endpoint, or 
coincide otherwise. If exists, then their \emph{join} $X*Y$ 
is given by $\{tx+(1-t)y\mid x\in X,\ y\in Y,\ t\in [0,1]\}$.
\begin{prop}\label{lem:neigh}
  With the triangulation $u_{\varphi'}$, the link of
  a vertex in $(D^1,S^0)^K$ is either $|K'|$, or is 
  $\mathrm{PL}$ homeomorphic to the join 
  \begin{equation}
	\underbrace{S^0*S^0*\cdots*S^0}_{k+1}*|\mathrm{Lk}( \sigma',K')|,
	\label{eq:join}
  \end{equation}
  where  $\sigma'\in K'$ is a $k$-simplex ($k\geq 0$),  
  $\mathrm{Lk}( \sigma',K')$ the link of $\sigma'$ in $K'$.
\end{prop}
\begin{proof}[Proof of Theorem \ref{thm:Davis}]
  First if $n\leq 2$, then $|K|$   can always
  be realized as the boundary of a convex polytope in $\mathbb{R}^n$, thus
  the statement follows from Remark \ref{rem:smooth}. 

  Suppose that $n\geq 3$. It is easy to see that the 
  link of any $k$-simplex ($k\geq 0$) in a homology $(n-1)$--manifold 
  (resp.~a $\mathrm{PL}$ $(n-1)$--manifold) is a 
  generalized homology $(n-k-2)$--sphere 
  (resp.~a $\mathrm{PL}$ $(n-k-2)$--sphere).\footnote{For instance, by induction
	on dimension $k$, with the observation that
	the link of a simplex $\sigma$ in the manifold is 
	that of a vertex in the link of a codimension-$1$ face
	of $\sigma$.}
	In particular, $|\mathrm{Lk}(\sigma',K')|$ is connected if $k=0$,
	thus the space \eqref{eq:join} is always simply connected.
  Together with the fact that the exterior join of a generalized homology 
  sphere (resp.~a $\mathrm{PL}$ sphere) with $S^0$ will be a sphere
  of the same type,\footnote{the first case is easy; 
  see \cite[Proposition 2.23, pp.~23--24]{RS72} for
  the $\mathrm{PL}$ case}
  the statement follows
  from Proposition \ref{lem:neigh}, and the Edwards--Freedman Theorem. 
  \end{proof}
\begin{proof}[Proof of Proposition \ref{lem:neigh}]
  Since $\mathbb{Z}_2^m$ acts on $\mathcal{U}(\mathbb{Z}_2^m,|K'_+|)$ 
  simplicially, it suffices to consider the links of each vertex
  in the image of $\varphi'$. The link of $\varphi'(|\emptyset|)$
  is $|K'|$. In what follows we consider other links.

  Consider a vertex $(\sigma)\in K'_+$, where $\sigma\in K$ is 
  a $k$-simplex, $k\geq0$.
  Let $K_{<\sigma}'$ (resp.~$K'_{\geq \sigma}$) 
  be the subcomplex of $K'_+$ consisting 
  of chains that ends with a \emph{proper face} of $\sigma$ (resp.~\emph{begins} with
  $\sigma$). By definition, 
  $\varphi'(|K_{<\sigma}'|)*\varphi'(|K'_{\geq \sigma}|)$ 
  exists, which is $\varphi'(|K_{<\sigma}'*K'_{\geq \sigma}|)$, 
  where we see that $K_{<\sigma}*K'_{\geq \sigma}$ is the subcomplex 
  of chains containing $\sigma$, i.e., it is the star of $(\sigma)$ in $K'_+$. 
  It follows that $|K'_{\geq \sigma}|=\bigcap_{i\in\sigma}Y_i$, where
  $\varphi'(Y_i)$ is fixed by $s_i$ (see \eqref{def:U}), thus 
  the star of $|(\sigma)|$ in 
  $\mathcal{U}(\mathbb{Z}_2^m,|K'_+|)$ is given by
  \begin{equation}
	\bigcup_{g\in\langle s_i\rangle_{i\in\sigma}}
	g|K_{<\sigma}'*K'_{\geq \sigma}|.
	\label{eq:nei1}
  \end{equation}
  Since $\varphi(|K'_{\geq\sigma}|)$ is invariant under the 
  subgroup $\langle s_i\rangle_{i\in\sigma}$, the image
  of \eqref{eq:nei1} under $\varphi'$ is 
  \begin{equation}
	\bigcup_{g\in\langle s_i\rangle_{i\in\sigma}}
	g\varphi'(|K_{<\sigma}'|)*\varphi'(|K'_{\geq \sigma}|)=                           
    \left(\bigcup_{g\in\langle s_i\rangle_{i\in\sigma}}
	g\varphi'(|K_{<\sigma}'|)\right)*\varphi'(|K'_{\geq \sigma}|).
	\label{eq:nei2}
  \end{equation}
  Let $\sigma'\in K'$ be any $k$-simplex (as a chain of length $k+1$) 
  ending with $\sigma$. Then   
  \begin{equation}
	\varphi'(|K'_{\geq \sigma}|)=
  \varphi'(|(\sigma)|)*\varphi'(|\mathrm{Lk}(\sigma',K')|)=
  \varphi'(|(\sigma)|)*|\mathrm{Lk}(\sigma',K')|.
	\label{eq:nei3}
  \end{equation}
  At last, let $K_{\leq \sigma}'$ be the subcomplex of $K'_+$ consisting 
  of chains ending with a face of $\sigma$. It is straightforward to see
  that
  $\widetilde{C}_{\sigma}=\bigcup_{g\in\langle s_i\rangle_{i\in\sigma}}g\varphi'(|K_{\leq \sigma}'|)$
  is a $(k+1)$-cube centered at $\varphi'( |(\sigma)| )$, i.e., 
  $C_{\sigma}$ in \eqref{def:cube} with $[0,1]$ replaced by $[-1,1]$.
  It follows that $\bigcup_{g\in\langle s_i\rangle_{i\in\sigma}}
	g\varphi'(|K_{<\sigma}'|)$ is the boundary of $\widetilde{C}_{\sigma}$,
	hence it is a $k$--sphere $\mathrm{PL}$ homeomorphic to the joins
	of $k+1$ copies of $S^0$. Together with \eqref{eq:nei3} and \eqref{eq:nei2}, 
	Proposition \ref{lem:neigh} follows. 
  \end{proof}
  \subsection{An application of Theorem \ref{thm:Davis}}\label{subs:app1}
  Suppose the vertex set of $K$ is $[m]$. A subset $\tau\subset [m]$ not
  contained in $K$ is called a \emph{missing face}, if any proper
  subset of $\tau$ is a simplex of $K$.   
  Clearly $K$ is determined by its missing faces.
  
  Let $J=(j_k)_{k=1}^m$ be an $m$-tuple of positive integers, with 
  $d(J)=\sum_{k=1}^mj_k$. For $i=1,2,\ldots,m$, let $B_i$ be the 
  block of $j_i$ integers (we set $j_0=0$)
  \[\sum_{k=1}^{i-1}j_k+1,	\ \sum_{k=1}^{i-1}j_k+2, \ \ldots, \ \sum_{k=1}^{i}j_k.\]

  \begin{defin}[cf.~\cite{BBCG10b}]\label{def:BBCG}
	Suppose that $\tau=\{i_1,i_2,\ldots,i_l\}$ is a missing face
	of $K$. Denote by $\tau (J)$ 
	the set $\{B_{i_1},B_{i_2},\ldots, B_{i_l}\}$. With respect to 
	$K$ and $J$,
	we define an abstract simplicial complex $K(J)$	
	in such a way that, as $\tau$ runs through all missing faces
	of $K$, $\tau(J)$ gives all missing faces of $K(J)$. The vertex
	set of $K(J)$ is $[d(J)]$.\footnote{The definition here is slightly different from
	  \cite[Definition 2.1]{BBCG10b}: we identify the 
	vertices of $K(J)$ with $[d(J)]$ a priori and explicitly.}
  \end{defin}
  In what follows, let $(\underline{D},\underline{S})$ be 
  the pairs $(D^{j_i},S^{j_i-1})$, i.e., the unit $j_i$-disk with its 
  boundary, $i=1,2,\ldots, m$.
  \begin{lem}[cf.~\cite{BBCG10b}]\label{lem:BBCG}
	The polyhedral product $(\underline{D},\underline{S})^K$ is 
	homeomorphic to the real moment-angle complex $(D^1,S^0)^{K(J)}$
	(see \eqref{def:pp} for definition).
  \end{lem}
  \begin{proof}
	Let $f=(f_1,f_2,\ldots,f_m)$ be the product of canonical homeomorphisms 
	$f_i\co I^{j_i}\to D^{j_i}$, $I=[-1,1]$. We write	
	$f\co I^{d(J)}\to \prod_{i=1}^mD^{j_i}$, which is clearly a homeomorphism.
	We claim that the restriction of $f$ to $(D^1,S^0)^{K(J)}$ 
	induces a homeomorphism onto $(\underline{D},\underline{S})^K$.
	By \eqref{def:pp},
	$x=(x_i)_{i=1}^m\in (\underline{D},\underline{S})^K$ if and only if 
	$\sigma_x=\{i\in [m]\mid x_i\in D^{j_i}\setminus S^{j_i-1}\}$ does
	not contain any missing faces of $K$. This happens if and only if
	$\sigma_{x}(J)=\{B_i \mid f^{-1}_i(x_i)\in I^{j_i}
	\setminus \partial I^{j_i}\}$ does not contain any missing faces of 
	$K(J)$. Therefore, the claim holds, from which the statement follows.
  \end{proof}
  Following \cite{BBCG10b}, $K(J)$ can be understood in another way. 
  
  Recall that the \emph{simplicial join} of two disjoint 
  complexes $K_1$ and $K_2$ is the complex  
  $K_1*K_2=\{\sigma_1\cup\sigma_2\mid\sigma_i\in K_i,i=1,2\}$.
  Suppose that $v_i=\{i\}$ is the $i$-th vertex of 
  $K$, $i=1,2,\ldots,m$. Let $K(v_i)$ be the 
  \emph{simplicial wedge} of $K$ on $v_i$, which 
  is a simplicial complex given by
  \begin{equation}
	K(v_i)=\{i\}*K_{[m]\setminus\{i\}}\bigcup
	\{i+1\}*K_{[m]\setminus\{i\}}\bigcup\{i,i+1\}*\mathrm{Lk}(v_i,K),	
	\label{def:sw}
  \end{equation}
  where $K_{[m]\setminus\{i\}}=\{\sigma\in K\mid\sigma\subset [m]\setminus\{i\}\}$ 
  is the \emph{full subcomplex} 
  containing the link $\mathrm{Lk}(v_i,K)$,
  such that for all $j> i$, the label of the original $j$-th vertex is 
  shifted to $j+1$ (labels $\leq i$ are preserved). Thus the vertex 
  set of $K(v_i)$ is $[m+1]$.
  
  For $i=1,2,\ldots,m$, let $J_i=(j_k)_{k=1}^m$ be the 
  tuple with $j_i=2$, and $j_k=1$ for all $k\not=i$.
  One can check that $K(J_i)=K(v_i)$. In this way, $K(J)$
  can be obtained via consecutive simplicial wedge 
  constructions. Moreover, it can be shown that $|K(v_i)|$ is $\mathrm{PL}$
  homeomorphic to the suspension $S^0*|K|$ 
  (see for example, \cite[Proposition 2.2]{CP15}). As a conclusion, we have
  the lemma below, the proof of which is omitted 
  (see the proof of Theorem \ref{thm:Davis}).
  \begin{lem}
	$|K(v_i)|$ is a generalized homology sphere (resp.~a $\mathrm{PL}$ sphere) 
	if and only if $|K|$ is (see Definition \ref{def:sphere}). 
  \end{lem}
  \begin{prop}
	The polyhedral product $(D^1,S^0)^{K(J)}$ is 
	a homology manifold (resp.~a $\mathrm{PL}$ manifold) 
	of dimension $n+d(J)-m$, 
	if and only if $|K|$ is a generalized homology sphere
	(resp.~a $\mathrm{PL}$ sphere) 
	of dimension $n-1$. Moreover, $(D^1,S^0)^{K(J)}$ is 
	a topological manifold if additionally $d(J)>m$.
  \end{prop}
  \begin{proof}
	Since every simplicial wedge construction increases the dimension 
	by one,
	the first statement follows directly from Theorem \ref{thm:Davis},
	and the lemma above. For the second one, from the Van Kampen theorem, 
    we only need to consider the case when $|K|$ is the $0$--sphere ($m=2$), 
	with $d(J)-m=1$.	
	In this case $|K(J)|$ bounds the $2$-simplex, 
	thus $(D^1,S^0)^{K(J)}$ is the $2$--sphere.
  \end{proof}
  In particular, suppose that all pairs are taken as $(D^2,S^1)$. 
  Together with Lemma \ref{lem:BBCG}, we have 
  (cf.~{\cite[Problem 6.14]{BP02}):
  \begin{cor}
	The moment-angle complex $(D^2,S^1)^K$ is a topological $(n+m)$--manifold
	if and only if $|K|$ is a generalized homology $(n-1)$--sphere.
  \end{cor}
  \section{Cup and cap products}\label{sec:Whitney}
   In this section we consider the cohomology of a real moment-angle complex
   $(D^1,S^0)^K$, with products involved. 
   
   \subsection{Whitney's formulae}\label{subs:Wf}
   We say that an abstract simplicial complex $K$ is of \emph{finite type}, if
   the number of simplices in each dimension is finite. 
    Let $C_*(K)$ (resp.~$C^*(K)$) denote the simplicial chain complex 
   (resp.~cochain complex) of $K$, where 
   $C_*(K)=\bigoplus_{i=0}^{\infty}C_i(K)$ 
   (resp.~$C^*(K)=\bigoplus_{i=0}^{\infty}C^i(K)$).  
   If $K$ is of finite type, then each $C_i(K)$ is finitely
   generated, whose dual basis generates
   $C^i(K)=\mathrm{Hom}(C_i(K),\mathbb{Z})$.
   
   Suppose that $\bm{p}=(p_i)_{i=1}^m$ and 
   $\bm{q}=(q_i)_{i=1}^m$ are two vectors of 
   integers.   The notation $(\bm{p},\bm{q})$ means 
   the $\mathrm{mod}$ $2$ integer from the shuffle
   \[ p_1,p_2,\ldots,p_m,q_1,q_2,\ldots,q_m \to p_1,q_1,p_2,q_2,\ldots,p_m,q_m,\]
   namely by interchanging adjacent integers to make the left sequence into 
   the right one, each interchange yields a summand, being the product of the 
   two integers involved, and 
   $(\bm{p},\bm{q})$ is the sum of these summands. 
   A straightforward calculation shows that 
   \begin{equation}
	 (\bm{p},\bm{q})=\sum_{i=1}^mq_i\sum_{j>i}p_j  \quad \mathrm{mod} \ 2.
	 \label{def:pq}
   \end{equation}
	 
   Let $X=\prod_{i=1}^m|K_i|$ be a product of polyhedra, 
	 with each $K_i$ of finite type. Denote by 
	 $C_*(X)=\bigoplus_{p=0}^{\infty}C_p(X)$ 
	 the tensor product $\bigotimes_{i=1}^mC_{*}(K_i)$, 
	 which is a differential graded $\mathbb{Z}$-module 
	 with the boundary operator $\partial$ subject to
	 \begin{equation}
	   \partial(\otimes_{i=1}^mc_{p_i})=
	   (-1)^{\sum_{j<i}p_j}c_{p_1}\otimes \ldots\otimes 
	   c_{p_{i-1}}\otimes \partial c_{p_i}\otimes c_{p_{i+1}}\otimes\ldots\otimes c_{p_m},
	   \label{eq:bound1}
	 \end{equation}
	   where $c_{p_i}\in C_{p_i}(K)$.
	   Let $(C^*(X),\mathrm{d})$ be the \emph{graded dual} of $(C_*(X),\partial)$. 
	   By assumption, we
	   have \[C^p(X)=\bigoplus_{\sum_{i=1}^m p_i=p}\bigotimes_{i=1}^mC^{p_i}(K_i).\]	
	   The notation $\left(C_*(X),C^*(X),\smile,\frown\right)$ means that	
	   $C_*(X)$ and $C^*(X)$ are endowed with products
	   \[\smile\co C^*(X)\otimes C^*(X)\to C^*(X) \quad \text{and} \quad 
		 \frown\co C^*(X)\otimes C_*(X)\to C_*(X)\] respectively, 
	   such that 
	   \begin{equation}
		 \left(\otimes_{i=1}^m c^{p_i}\right)\smile 
		 \left(\otimes_{i=1}^mc^{q_i}\right)
		 =(-1)^{(\bm{p},\bm{q})}
		 \otimes_{i=1}^m c^{p_i}\smile c^{q_i}
		 \label{eq:cup}
	   \end{equation}
		 with $c^{p_i}\in C^{p_i}(K_i)$, $c^{q_i}\in C^{q_i}(K_i)$,
		 $\bm{p}=(p_i)_{i=1}^m$ and $\bm{q}=(q_i)_{i=1}^m$, and 
		 \begin{equation}
		   \left(\otimes_{i=1}^m c^{p_i}\right)\frown 
		   \left(\otimes_{i=1}^mc_{r_i}\right)=(-1)^{(\bm{r}-\bm{p},\bm{p})}
		   \otimes_{i=1}^mc^{p_i}\frown c_{r_i}
		   \label{eq:cap}
		 \end{equation}
 with $c_{r_i}\in C_{r_i}(K_i)$ and $\bm{r}=(r_i)_{i=1}^m$, 
 where each product $\smile$ and $\frown$ on the right-hand
 sides of \eqref{eq:cup} and \eqref{eq:cap} means the \emph{simplicial
 cup and cap products}, respectively 
 (see \eqref{def:cup0} and \eqref{def:cap0}).
	 
	 We treat $X$ as a $CW$ complex,
	 with each cell of the form $\prod_{i=1}^m|\sigma_i|$, $\sigma_i\in K_i$.
	 Let $A$ be a subcomplex of $X$. Denote by $C_*(A)$ the restriction of 
	 $C_*(X)$ to $A$; clearly $C_*(A)$ is closed under $\partial$. 
	 Let $\upsilon_c\co C_*(A)\to C_*(X)$ be the chain inclusion, and
	 let 
	 $(C^*(A),\mathrm{d})$ be the graded dual of $(C_*(A),\partial)$. 
	 Since the dual $\upsilon_c^*$ is surjective, 
	 the quadruple $\left(C_*(A),C^*(A),\smile,\frown\right)$ can be defined
	 uniquely, such that diagrams
    \[
	  \begin{CD}
	   C^*(X)\otimes C^*(X)@>\smile>> C^*(X)\\
       @V\upsilon_c^*\otimes\upsilon_c^*VV   @V\upsilon_c^*VV\\
       C^*(A)\otimes C^*(A)@>\smile>> C^*(A)
	  \end{CD}
	  \quad \text{and}\quad
	  \begin{CD}
	   C^*(X)&\otimes& C_*(X)@>\frown>> C_*(X)\\
       @V\upsilon_c^*VV  @A\upsilon_cAA   @A\upsilon_c AA\\
       C^*(A)&\otimes& C_*(A)@>\frown >> C_*(A)
	  \end{CD}
\]
commute.
 \begin{thm}\label{thm:Whitney}
   Let $\left(H_*(A),H^*(A),\smile,\frown\right)$ be the quadruple 
   of singular (co)homology of $A$, endowed with cup and cap products.	
   Then we have an isomorphism
   \begin{equation}
	 \left(H_*(C_*(A),\partial),H^*(C^*(A),\mathrm{d}),\smile,\frown\right)
	 \cong\left(H_*(A),H^*(A),\smile,\frown\right)
	 \label{iso:prod}
   \end{equation}
   on passage to (co)homology, preserving the products. 
 \end{thm}
 \begin{rem}
   The quadruple 
   $\left(H_*(C_*(X),\partial),H^*(C^*(X),\mathrm{d}),\smile,\frown\right)$
   with formulae \eqref{eq:cup} and \eqref{eq:cap} were given in 
   Whitney \cite[pp.~424--426]{Whi38}.  
 \end{rem}
 Theorem \ref{thm:Whitney} will be illustrated in what follows; 
 the proof will given in Section \ref{proof:Whitney}.

\subsection{The (co)homology of a real moment-angle complex}\label{ss:rmac}
   Henceforth, occasionally we will not distinguish a geometric cell 
   (resp.~geometric simplex) or the associated cellular chain 
   (resp.~simplicial chain). 
   
 Let $I=[-1,1]$ be a simplicial complex with a single 
 $1$-simplex $u$ connecting the two endpoints 
 $\underline{t}=\{-1\}$ and $t=\{1\}$, respectively,
 and they are obviously oriented that 
 $\partial u=t-\underline{t}$. 
 Let $I^m$ be the $CW$-complex, with each cell an $m$-fold 
 product of simplices of the form above.
 
 By \eqref{def:pp}, $(D^1,S^0)^K$ is a cellular
 subcomplex embedded in $I^m$, such that a cell
 $e=\prod_{i=1}^me_i\subset I^m$ belongs to $(D^1,S^0)^K$
 if and only if $\sigma_e=\{i\in[m]\mid e_i=u\}$ is a simplex
 of $K$. 
 
 In the remainder of this section, let $(X,A)$ be the pair 
 $(I^m,(D^1,S^0)^K)$.

 Now we perform a change of basis for $C_*(A)$ and $C^*(A)$,
 respectively.
 For $\left(C_*(I),\partial\right)$ 
 indicated above,
 denote by 
 \begin{equation}
   \varepsilon=\partial u=t-\underline{t}, 	
   \label{def:epsilon}
 \end{equation}
 then $C_*(I)$ is generated by $u$, $\varepsilon$ and $\underline{t}$.
 After dualizing, for $\left( C^*(I),\mathrm{d}\right)$,
 we use basis elements $u^*$, $t^*$ and 
 \begin{equation}
   \delta^*=t^*+\underline{t}^*,
   \label{def:delta}
 \end{equation}
	 where $u^*$, $\underline{t}^*$ and $t^*$
	 are the dual simplices of $u$, $\underline{t}$ and $t$, 
	 respectively. Immediately we have $\mathrm{d}t^*=u^*$ and 
	 $\mathrm{d}u^*=0=\mathrm{d}\delta^*$.  

	 \begin{defin}\label{def:bases}
	 Given a basis element $c=\otimes_{i=1}^m c_i\in C_*(X)$, $c_i\in C_*(I)$,
	 we define $\sigma_c=\{i\in [m]\mid c_i=u\}$ and 
	 $\tau_c=\{i\in [m]\mid c_i=\varepsilon\}$.
	 Then we denote  by $c$ the word
	 \[u_{\sigma_c}\varepsilon_{\tau_c}\underline{t}_{[m]\setminus(\sigma_c\cup\tau_c)},\] 
	 which we shall abbreviate as $u_{\sigma_c}\varepsilon_{\tau_c}$, omitting the obvious 
	 part $\underline{t}_{[m]\setminus(\sigma_c\cup\tau_c)}$. Clearly
	 each abbreviated word $u_{\sigma}\varepsilon_{\tau}$ with
	 $\sigma\cap\tau=\emptyset$ 
	 corresponds uniquely to a chain of $C_*(X)$.
     
	 Analogously, we denote a dual basis element $c=\otimes c^{i}\in C^*(X)$, $c^i\in C^*(I)$,
	 by the word
	 \[ u^{\sigma_c}t^{\tau_c}\delta^{[m]\setminus (\sigma_c\cup\tau_c)},\]
	 with $\sigma_c=\{i\in [m]\mid c^i=u^*\}$ and $\tau_c=\{i\in [m]\mid c^i=t^*\}$,
	 and write it as $u^{\sigma_c}t^{\tau_c}$.
	 \end{defin}
	  By definition, we see that $C_*(A)$ and $C^*(A)$ are generated by 
	 $\{u_{\sigma}\varepsilon_{\tau}\mid \sigma\in K, \sigma\cap\tau=\emptyset\}$
	 and
	 $\{u^{\sigma}t^{\tau}\mid \sigma\in K, \sigma\cap\tau=\emptyset\}$ 
	 respectively,
	 with the \emph{void word} corresponding to $\underline{t}_{[m]}$ or 
	 $\delta^{[m]}$.

	 For $i\in [m]$, we use the notation 
	 $(i,\sigma)=\mathrm{card}\{j\in\sigma\mid j<i\}$. 
	 Now take $u_{\sigma}\varepsilon_{\tau}\in C_*(A)$ and $u^{\sigma}t^{\tau}\in C^*(A)$ 
	 respectively,
	 by \eqref{eq:bound1} and its dual form, 
	 we have 
	 \begin{equation}
	   \partial(u_{\sigma}\varepsilon_{\tau})=
	   \sum_{i\in\sigma}(-1)^{(i,\sigma)}u_{\sigma\setminus\{i\}}\varepsilon_{\tau\cup\{i\}},
	   \label{eq:bound2}
	 \end{equation}
	and 
	\begin{equation}
	   \mathrm{d}(u^{\sigma}t^{\tau})=
	   \sum_{\substack{(\sigma\cup\{i\})\in K\\i\in\tau }}
	   (-1)^{(i,\sigma)}u^{\sigma\cup\{i\}}t^{\tau\setminus\{i\}}.
	   \label{eq:cobound1}
	 \end{equation}
	 Recall that for each $\omega\subset [m]$, 
	 the \emph{full subcomplex} $K_{\omega}$ is given by 
	 $\{\sigma\in K\mid\sigma\subset\omega\}$. 
	 Consider the augmented chain complex $(\widetilde{C}_*(K_\omega),\partial)$ 
	 and its dual $(\widetilde{C}^*(K_{\omega}),\mathrm{d})$.
	 Each simplex $\sigma\in K_{\omega}$ can 
   be treated as a subset of $[m]$ equipped with the natural ordering 
   (thus $|\sigma|$ is positively oriented if the indices contained in 
   $\sigma$ are written in the increasing order), 
   then the simplicial boundary operator
   satisfies 
   \begin{equation}
	 \partial(\sigma)=\sum_{i\in\sigma}(-1)^{(i,\sigma)}\sigma\setminus\{i\},
	 \label{eq:bound3}
   \end{equation}
   together with $\partial(\{i\})=\emptyset_{\omega}$ 
	 for all $\{i\}\in K_{\omega}$, due to the augmentation.
   After dualizing, for a dual simplex $\sigma^*\in \widetilde{C}^*(K_{\omega})$, 
   we have $\mathrm{d}(\emptyset_{\omega})=\sum_{i\in\omega}\{i\}$, 
   and otherwise
   \begin{equation}
	 \mathrm{d}(\sigma^*)=\sum_{\substack{(\sigma\cup\{i\})\in K\\i\in\omega\setminus\sigma }}
	 (-1)^{(i,\sigma)}\left(\sigma\cup\{i\}\right)^*.
	 \label{eq:cobound2}
   \end{equation}
   
   A comparison of \eqref{eq:bound2} and \eqref{eq:bound3} 
   (resp.~\eqref{eq:cobound1} and \eqref{eq:cobound2}) yields the 
   following:
   \begin{prop}\label{prop:iso1}
	The mapping
	\[\mu\co\bigoplus_{\omega\subset [m]}\widetilde{C}_*(K_{\omega})\longrightarrow C_*(A)\]
	 given by sending each $\sigma_\omega\in\widetilde{C}_*(K_{\omega}) $ to 
	 the word $u_{\sigma}\varepsilon_{\omega\setminus\sigma}$ yields a chain isomorphism
	 that shifts the degrees up by one.
 Analogously, we have the degree-shifting cochain isomorphism
 \[ \eta\co \bigoplus_{\omega\subset [m]}\widetilde{C}^*(K_{\omega})
 \stackrel{\cong}{\longrightarrow} C^*(A) \]
 sending each $\sigma_{\omega}^*\in \widetilde{C}^*(K_{\omega})$ to 
 $u^{\sigma}t^{\omega\setminus\sigma}$.    
 \end{prop}
  Combined with Theorem \ref{thm:Whitney}, we have 
 isomorphisms
 \begin{equation}
   H_{p}(C_*(A), \partial)\cong
	   \bigoplus_{\omega\subset [m]}\widetilde{H}_{p-1}(K_{\omega})
	   \quad \text{and} \quad
	   H^{p}(C^*(A),\mathrm{d})\cong
	   \bigoplus_{\omega\subset [m]}\widetilde{H}^{p-1}(K_{\omega})
   \label{eq:hd}
 \end{equation}
in each dimension $p\geq 0$.
%%%%\begin{thm}\label{thm:cupcap}
%%%%On passage to (co)homology, we have a degree-shifting isomorphism
%%%%\[ \left( H_*(\widetilde{M}_*(K),\partial),H^*(\widetilde{M}_*(K),
%%%%  \mathrm{d}),\smile',\frown'\right)
%%%%  \cong \left( H_*(A), H^*(A),\smile,\frown\right)               \] 
%%%%  preserving the products. More explicitly, we have
%%%%   in each dimension $p\geq 0$.
%%%%\end{thm}

	Recall that, by definition, $\widetilde{H}_{-1}(K_{\omega})$ is 
	non-trivial only when $\omega=\emptyset$; 
	$\widetilde{H}_{p-1}(K_{\emptyset})$ vanishes if $p>0$, 
	and is infinite cyclic if $p=0$. 
	Therefore, we see that 
	$\widetilde{H}_{-1}(K_{\emptyset})$
	corresponds to $H_0(A)$; the representative of $\widetilde{H}_{-1}(K_{\emptyset})$ 
	is the ``empty set'' of $K_{\emptyset}$, 
	which is sent by $\mu$ to the void word associated to   
	$\underline{t}_{[m]}$, i.e., the point in $I^m$ with constant 
	coordinates $-1$.
     
\section{Applications of Whitney's formulae}\label{app:Whitney}
In order to apply Whitney's formulae \eqref{eq:cup} and \eqref{eq:cap} 
to the cup and cap products in a real moment-angle complex, by
Theorem \ref{thm:Whitney}, it remains 
to understand the products in the simplicial complex $I=[-1,1]$.

 Recall that the \emph{simplicial cup and cap products} 
in a simplicial complex $|K|$, i.e., 
$\smile\co C^*(K)\otimes C^*(K)\to C^*(K)$ and 
$\frown\co C^*(K)\otimes C_*(K)\to C_*(K)$,
are given as follows. Choose a partial ordering on the vertex set of $K$ which
induces a total ordering on each simplex, and let 
$[v_{i_0},v_{i_1},\ldots,v_{i_p}]\in C_*(K)$ denote
such a $p$-simplex, with $v_0<v_1<\ldots<v_p$ in the given ordering.
Then we have 
\begin{equation}
   (c^p\smile c^q)([v_{i_0},\dots,v_{i_{p+q}}])=
 c^p([v_{i_0},\dots,v_{i_p}])
 c^q ([v_{i_p},\dots,v_{i_{p+q}}])
   \label{def:cup0}
 \end{equation}
and
\begin{equation}
  c^p\frown [i_0,\dots,i_r]=c^p([i_{r-p},\dots,i_r])
[i_0,\dots,i_{r-p}] \quad (c^p([i_{r-p},\dots,i_r])\in\mathbb{Z}),
\label{def:cap0}
\end{equation}
for $c^p\in C^p(K)$ and $c^q\in C^q(K)$.

For instance, let $(C_*(I),\partial)$ 
and $(C^*(I),\mathrm{d})$ be the 
simplicial (co)chain complexes given in Subsection \ref{ss:rmac},
endowed with bases $\{u, \varepsilon, \underline{t}\}$ and 
$\{u^*,\delta^*, t^*\}$, respectively
(see \eqref{def:epsilon}, \eqref{def:delta}). It can be easily
checked that the following formulae hold:
\begin{align*}
   &t^*\smile t^*= t^*,\quad t^*\smile u^*=0,\quad 
   u^*\smile t^*=u^*, \quad u^*\smile u^*=0, \\
   &\delta^*\smile u^*=u^*\smile \delta^*=u^*, \quad 
   \delta^*\smile t^*=t^*\smile\delta^*=t^*, \quad 
   \delta^*\smile\delta^*=\delta^*;\\
	&	\delta^*\frown \underline{t}=\underline{t},
	\quad u^*\frown u=\underline{t},\quad u^*\frown\varepsilon=0,
	\quad u^*\frown \underline{t}=0,\\
   &	t^*\frown\varepsilon=\varepsilon+\underline{t}, \quad 
   t^*\frown u=u,\quad t^*\frown \underline{t}=0,\quad 
	\delta^*\frown\varepsilon=\varepsilon,\quad \delta^*\frown u=u.
  \end{align*}
  
  For $\sigma\subset [m]$, let $\bm{v}(\sigma)=(v_i)_{i=1}^m\in\mathbb{Z}^m$ be the  
  vector with $v_i=1$ if $i\in\sigma$, and $v_i=0$ otherwise.
  By \eqref{eq:cup}, \eqref{eq:cap} and the products presented above, 
  choosing two cochains $u^{\sigma}t^{\tau}$ and $u^{\sigma'}t^{\tau'}$ 
  from $C^*( (D^1,S^0)^K)$ (see Definition \ref{def:bases}), we have (see 
  \eqref{def:pq} for notations)
  \begin{equation}
	u^{\sigma}t^{\tau}\smile u^{\sigma'}t^{\tau'}=
	(-1)^{\left(\bm{v}(\sigma),\bm{v}(\sigma')\right)}u^{\sigma\cup\sigma'}
	  t^{\tau\cup(\tau'\setminus\sigma)}
	\label{eq:cup2}
  \end{equation}
  if $\sigma'\cap(\sigma\cup\tau)=\emptyset$ and $(\sigma\cup\sigma')\in K$,
  otherwise the product vanishes.
  Additionally, choosing $u_{\sigma''}\varepsilon_{\tau''}\in C_*( (D^1,S^0)^K)$, 
  then 
  \begin{equation}
	u^{\sigma}t^{\tau}\frown u_{\sigma''}\varepsilon_{\tau''}=
	(-1)^{\left(\bm{v}(\sigma''\setminus\sigma),\bm{v}(\sigma)\right)}
	  \displaystyle{\sum_{\gamma\subset(\tau\setminus\sigma'')}}
  u_{\sigma''\setminus\sigma}\varepsilon_{\tau''\setminus\gamma}, 	
	\label{eq:cap2}
  \end{equation}
   provided that $\sigma\subset \sigma''$ and $\tau\subset(\sigma''\cup\tau'')$; 
   otherwise $u^{\sigma}t^{\tau}\frown u_{\sigma''}\varepsilon_{\tau''}$ vanishes.
   Suppose that $\mathrm{card}(\sigma)=p$ and $\mathrm{card}(\sigma'')=r$, it
   turns out that (see \eqref{eq:h})
   \begin{equation}
	 \partial(u^{\sigma}t^{\tau}\frown u_{\sigma''}\varepsilon_{\tau''}) 
	 =(-1)^{r-p}\mathrm{d}(u^{\sigma}t^{\tau})\frown u_{\sigma''}\varepsilon_{\tau''}
	 +u^{\sigma}t^{\tau}\frown\partial( u_{\sigma''}\varepsilon_{\tau''}).
	 \label{eq:diffcap}
   \end{equation}
   
   It is convenient to consider 
   $\left(C^*( (D^1,S^0)^K),\smile\right)$ as a differential graded
   algebra with $2m$ generators, such that 
   $\delta^{[m]}$ is the unique identity (see Theorem \ref{thm:iso} for 
   more details).

 \begin{exm}
   Let $K$ be the pentagon with vertex set $[5]$, whose maximal simplices
  are 
  \[\{1,2\}, \ \{2,3\}, \ \{3,4\}, \ \{4,5\}, \ \{5,1\}.\]
  Using $\mathrm{mod}$ $5$ integers,
  it can be checked that full subcomplexes of $K$
with non-vanishing reduced (co)homology are: $K_{i,i+2}$ with $i=1,2,3$, 
$K_{i,i+3}$ with $i=1,2$, $K_{i,i+2,i+3}$ with $i=1,2,\ldots,5$, 
together with $K_{\emptyset}$ and $K_{[5]}$. Therefore by \eqref{eq:hd}, 
$H^{1}( (D^1,S^0)^K)$ 
 is torsion-free with ten generators 
 \begin{equation*}
   \alpha_1=[u^1t^3], \ \alpha_2=[u^1t^4],\ 
   \alpha_3=[u^2t^4],\ \alpha_4=[u^2t^5], \
   \alpha_5=[u^3t^5],
 \end{equation*}
and 
\begin{equation*}
  \beta_i=[u^it^{i+2}(1-t^{i+3})]=[u^it^{i+2}-u^it^{i+2,i+3}] 
  \quad i=1,2,\ldots, 5 \ \ \mathrm{mod} \ 5;
  \label{bas:beta}
\end{equation*}
we see that $H^{2}( (D^1,S^0)^K)$ 
is generated by $\gamma=[u^{j,j+1}t^{j+2,j+3,j+4}]$, 
with $j$ an arbitrary $\mathrm{mod}$ $5$ integer.
By Theorem \ref{thm:Davis}, $(D^1,S^0)^K$ is an orientable surface 
of genus $5$. Using \eqref{eq:cup2} we can give 
	a presentation of the cup products.
	Since $H^{1}(K_{\omega})$ is non-trivial only when 
	$\omega=\{1,2,3,4,5\}$, 
	a direct calculation  shows that 
	for all $\mathrm{mod}$ $5$ integers $i,j$,  
	$\alpha_i\smile\alpha_j=\beta_i\smile\beta_j=0$, then
	\[\gamma=-\alpha_1\smile\beta_2=\alpha_2\smile\beta_5=
  	 -\alpha_3\smile\beta_3=\alpha_4\smile\beta_1=-\alpha_5\smile\beta_4 \]
    presents all non-trivial cup products. 
  \end{exm}
  \begin{exm}[cf.~\cite{LdM89}, \cite{BLV13}, \cite{GL14}]\label{exm:LdM}
	Let $\lambda=(\lambda_i)_{i=1}^7$ be the $7$-tuple 
	with $\lambda_i\in\mathbb{R}^2$, 
  consisting of the real and imaginary parts of the 
  solutions of the equation $z^7=1$.
  Consider the variety $Z(\Lambda)\subset\mathbb{R}^7$ 
  given by the intersections of 
  quadrics 
  \[\begin{cases}
			\sum_{i=1}^7\lambda_ix_i^2=\bm{0},\\
		\sum_{i=1}^7x_i^2=1.
	  \end{cases}
	  \]
  It can be checked directly that $Z(\Lambda)$ is a 
  smooth $4$--manifold.\footnote{See for example, \cite[Lemma 0.3, p.~58]{BM06}.}

Let $K$ be a simplicial complex with vertex set $[7]$, such 
that $\sigma\subset [7]$ belongs to $K$ if and only if 
the origin of $\mathbb{R}^2$ is in the relative interior of the 
convex hull $\mathrm{conv}(\lambda_i)_{i\in [7]\setminus\sigma}$.
It can be shown that $|K|$ bounds a convex polytope, and 
there is a piecewise differentiable homeomorphism 
$(D^1,S^0)^K\to Z(\Lambda)$.\footnote{See for example, 
  \cite[Lemma 0.12, p.~64]{BM06} for the first statement, and 
  \cite[Lemma 6.3]{Cai15} for the second.} 

Observe that any subset of $[7]$ with cardinality 
 $2$ is a simplex of $K$, and 
any subset of $[7]$ not in $K$ must contain three consecutive points of the 
form $\{i,i+1,i+2\}$, $i=1,2,\dots,7$ $\mathrm{mod}$ $7$ 
(see Figure \ref{fig:7gon}).
Therefore, $(D^1,S^0)^K$ is  simply connected 
(see \cite[Chapter 1, p.~12]{Dav08}). 
Using \eqref{eq:hd}, we can write down 
the orientation class of $(D^1,S^0)^K$ through $H_3(K)$:
after listing all $14$ $3$-faces of $|K|$ and assign to each of
them a coefficient $\pm 1$ to make a cycle, we see that 
(see Proposition \ref{prop:iso1})
\begin{align*}
  \Gamma=&[u_{1,2,4,5}\varepsilon_{3,6,7}-u_{1,2,4,6}\varepsilon_{3,5,7}+
	u_{1,2,5,6}\varepsilon_{3,4,7}+
	u_{1,3,4,6}\varepsilon_{2,5,7}-
	u_{1,3,4,7}\varepsilon_{2,5,6}\\
	&-u_{1,3,5,6}\varepsilon_{2,4,7}+u_{1,3,5,7}\varepsilon_{2,4,6}-u_{1,4,5,7}\varepsilon_{2,3,6}
	+u_{2,3,5,6}\varepsilon_{1,4,7}-u_{2,3,5,7}\varepsilon_{1,4,6}\\
	&+
	u_{2,3,6,7}\varepsilon_{1,4,5}+u_{2,4,5,7}\varepsilon_{1,3,6}
	-u_{2,4,6,7}\varepsilon_{1,3,5}+u_{3,4,6,7}\varepsilon_{1,2,5}]. 
\end{align*}
\begin{figure}
  \begin{center}
          \includegraphics[width=7cm]{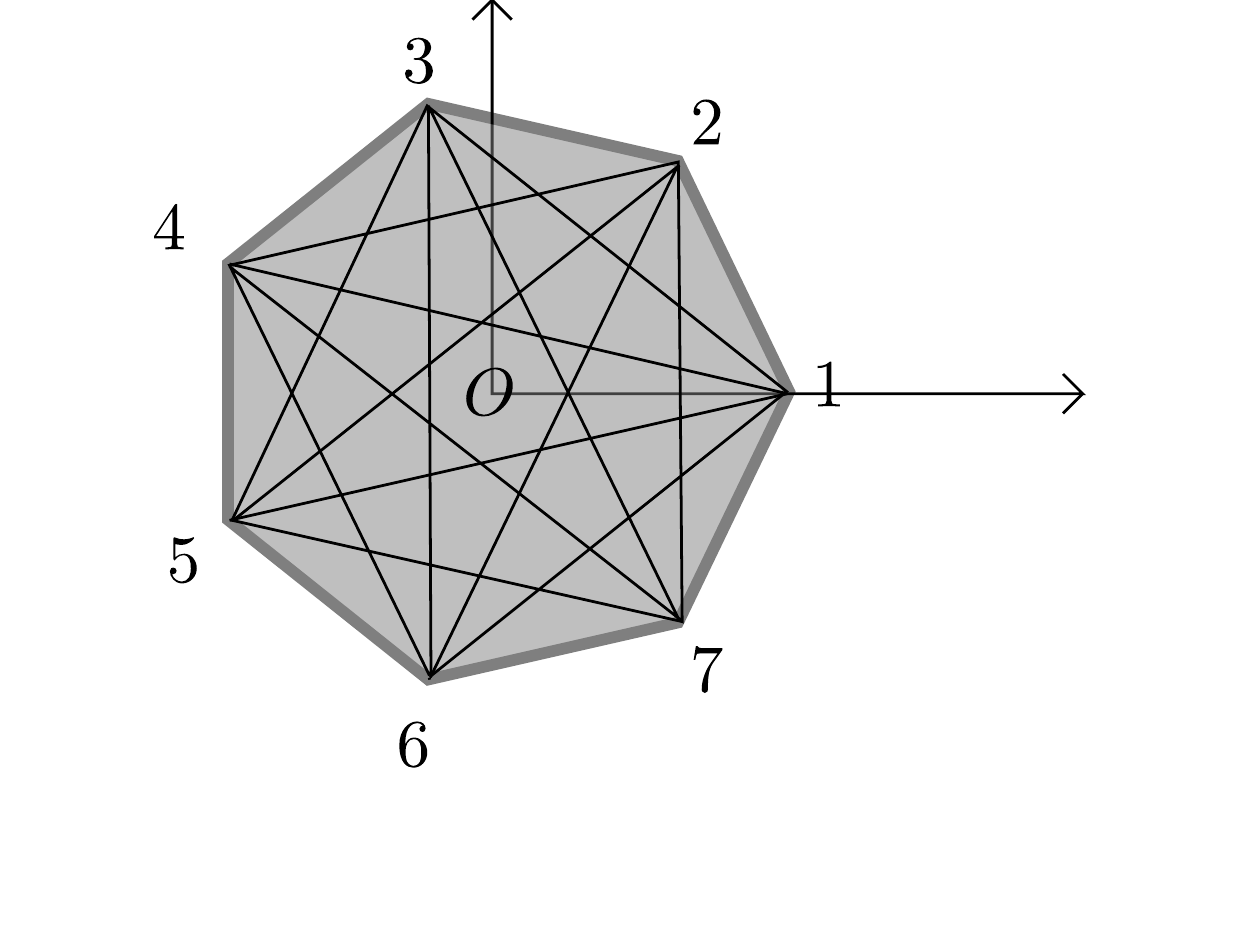}
  \end{center}
  \caption{The Heptagon}
  \label{fig:7gon}
\end{figure}
By Alexander duality, we have  
\begin{equation}
  \widetilde{H}^{3-i-1}(K_{\omega})\cong 
 \widetilde{H}_{i}(K_{[7]\setminus\omega}).
  \label{eq:AD1}
\end{equation}
 It follows that $\widetilde{H}^*(K_{\omega})$ is  
 non-trivial only when 
 $\mathrm{card}(\omega)=0,3,4,7$. In particular, 
  $\widetilde{H}^1(K_{\omega})$ is non-trivial if and 
  only if $\omega$ consists of three or four consecutive 
  $\mathrm{mod}$ $7$ integers (see Figure \ref{fig:fsubcpx}).
 
 As a conclusion,  
 $H^2( (D^1,S^0)^K)$ 
 is torsion-free with $14$ generators
 \[\alpha_i=[u^{i,i+1}t^{i+2}] \quad \text{and}\quad
 \beta_i=[u^{i+1,i+2}t^{i,i+3}],
  \]
 $i=1,2,\ldots,7$ $\mathrm{mod}$ $7$, according to $\mathrm{card}(\omega)=3$
 and $\mathrm{card}(\omega)=4$, respectively. 
 It can be checked that $\alpha_i\smile\alpha_j=0$,  
 and $\alpha_i\smile\beta_{j}$ is non-trivial if 
 and only if $j=i+3$, which generates $H^4((D^1,S^0)^K)$.  For instance, 
 $u^{1,2}t^{3}\smile u^{5,6}t^{4,7}=u^{1,2,5,6}t^{3,4,7}$.  Notice 
 that $\beta_i\smile\beta_{i+4}=\pm u^{i+1,i+2,i+5,i+6}t^{i+3,i+4,i+7}$, 
 which is non-trivial.
 If we change each $\beta_i$ into  
 \[\beta_i'=[u^{i+1,i+2}t^{i}(1-t^{i+3})]=[u^{i+1,i+2}t^{i}-u^{i+1,i+2}t^{i,i+3}],
 \] 
 then a straightforward calculation shows that each $\alpha_i$ has 
 a unique pairing $\beta'_{i+3}$ to make
 a non-trivial cup product, and $\beta_i'\smile\beta'_j$ vanish
 for all $\mathrm{mod}$ $7$ integers $i,j$.
 As a conclusion, we see that $(D^1,S^0)^K$ and $\sharp_{7}S^2\times S^2$ 
 have isomorphic cohomology rings, hence the two $4$--manifolds 
 are homeomorphic, by the classification 
 theorem of Freedman \cite{Fre82} 
 (an explicit diffeomorphism between them 
 is given by Guti\'{e}rrez and L\'{o}pez de Medrano 
 \cite{GL14}).
\begin{figure}
   \begin{center}
           \includegraphics[width=6cm]{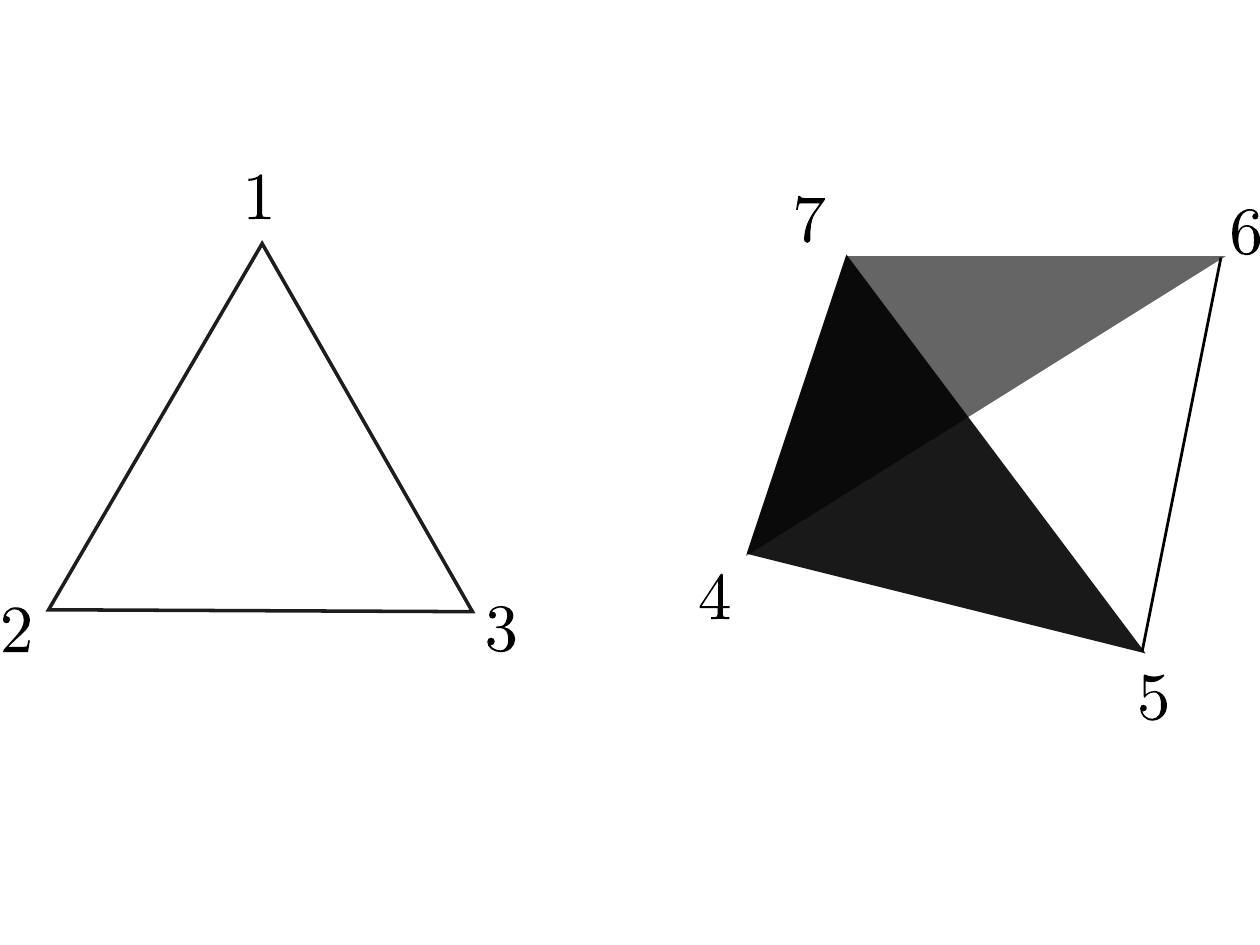}
   \end{center}
   \caption{Full subcomplexes $K_{1,2,3}$ and $K_{4,5,6,7}$}
   \label{fig:fsubcpx}
 \end{figure}

 By Poincar\'{e} duality, the intersection of submanifolds 
 can be understood through cup products. 
 For instance, by \eqref{eq:cap2}, we have
\begin{align} 
  [u^{1,2} t^3\frown\Gamma]&= [\underbrace{u_{4,5}\varepsilon_{6,7}-
	u_{4,6}\varepsilon_{5,7}+
	u_{5,6}\varepsilon_{4,7}}_{K_{4,5,6,7}}
	+\underbrace{u_{4,5}\varepsilon_{3,6,7}-
	u_{4,6}\varepsilon_{3,5,7}+
	u_{5,6}\varepsilon_{3,4,7}}_{K_{3,4,5,6,7}}]\label{eq:chop1}\\
    &=[\underbrace{u_{4,5}\varepsilon_{6,7}-
    u_{4,6}\varepsilon_{5,7}+
	u_{5,6}\varepsilon_{4,7}}_{K_{4,5,6,7}}], 
	\label{eq:chop2}
\end{align}
	because $K_{3,4,5,6,7}$ is acyclic. 
Here the geometric meaning of the 
class on the right-hand side of \eqref{eq:chop2} is as 
follows: consider the sphere 
\[ S=\{(x_i)_{i=1}^7\in (D^1,S^0)^K\mid x_1=x_2=x_3=-1, x_7=1\}\]
with suitable orientation, and let $s_7\co\mathbb{R}^7\to\mathbb{R}^7$ 
be the reflection changing the sign of the last coordinate, then 
$[u_{4,5}\varepsilon_{6,7}-
    u_{4,6}\varepsilon_{5,7}+
	u_{5,6}\varepsilon_{4,7}]$ corresponds to 
	the class $[S]+s_7[S]$ (see Definition \ref{def:bases}). 
	Therefore, if we use the representative
$[u^{1,2}t^3(1-t^7)]$ in \eqref{eq:chop1} instead of $[u^{1,2} t^3]$, 
we will get the class $s_7[S]$, which is represented by a submanifold 
whose intersection with the sphere
\[S'=\{(x_i)_{i=1}^7\in (D^1,S^0)^K\mid x_4=x_5=x_6=x_7=-1\}\]
is the point with constant coordinates $-1$. 
  In the same way, by expanding $[u^{5,6}t^{4}(1-t^7)\frown\Gamma]$ we
  see that it coincides with suitable orientation class $[S']$.
  From the Poincar\'{e} duality we can read the
  plumbing of spheres $s_7(S)$ and $S'$, which can also be checked 
  directly.
  \end{exm}

 \section{Cohomology of certain polyhedral products}\label{sec:coho}
Suppose $J=(j_i)_{i=1}^m$
is an $m$-tuple of positive integers. In this section we 
consider the relation between $H^*( (D^1,S^0)^K)$ 
and $H^*( (D^1,S^0)^{K(J)})$, where $K(J)$ is given in 
Definition \ref{def:BBCG}. 

 Let 
$\left(\mathbb{Z}\langle \wt{v}^i,\wt{u}^i\rangle_{i=1}^m, 
\mathrm{d}\right)$ be the 
free algebra generated by $2m$ generators
$\wt{v}^i$ with $\mathrm{deg}(\wt{v}^i)=j_i$ and 
$\wt{u}^i$ with $\mathrm{deg}(\wt{u}^i)=j_i-1$, respectively, 
$i=1,2,\ldots,m$;
for two homogeneous elements $x,y\in \mathbb{Z}\langle \wt{v}^i,\wt{u}^i\rangle_{i=1}^m$,
the differential $\mathrm{d}$ satisfies
\begin{equation}
  \mathrm{d}(xy)=(\mathrm{d}x)y+(-1)^{\mathrm{deg}(x)}x(\mathrm{d}y),
  \label{eq:Leib}
\end{equation}
with $\mathrm{d}\wt{u}^i=\wt{v}^i$ and $\mathrm{d}\wt{v}^i=0$.
Let $\left(R^*(J),\mathrm{d}\right)$ be the quotient of 
$\left(\mathbb{Z}\langle \wt{v}^i,\wt{u}^i\rangle_{i=1}^m, 
\mathrm{d}\right)$, subject to relations
 \begin{equation}
      x^ix^j=(-1)^{\mathrm{deg}(x^i)\mathrm{deg}(x^j)}x^jx^i, \quad
 (x^i)^2=\begin{cases}
   x^i & \text{if $\mathrm{deg}(x^i)=0$,}\\
   0 & \text{otherwise},
 \end{cases}
   \label{eq:relation1}
 \end{equation}
 where $x^i$ (resp.~$x^j$) is either $\widetilde{u}^i$ or $\widetilde{v}^i$ 
 (resp.~$\widetilde{u}^j$ or $\widetilde{v}^j$), for \emph{distinct} 
 $i,j=1,2,\ldots,m$, 
 together with
 \begin{equation}
   \widetilde{u}^i\widetilde{v}^i=0, \quad  
   \widetilde{v}^i\widetilde{u}^i=\begin{cases}
	 \wt{v}^i & \text{if $\mathrm{deg}(\wt{u}^i)=0$,}\\
	 0   & \text{otherwise},
 \end{cases} \quad i=1,2,\ldots,m.
   \label{eq:relation2}
 \end{equation}
 For a simplicial complex $K$ with vertex set $[m]$,
 the corresponding \emph{Stanley-Reisner ideal} $\mathcal{I}_{K}(J)\subset R^*(J)$ is 
 generated by all square-free monomials of the form
 $\widetilde{v}^{\tau}=\prod_{i\in\sigma}\wt{v}^{i}$, 
 where $\tau$ is \emph{not} a simplex of $K$. 

 Let $\left(R_K^*(J),\mathrm{d}\right)$ be the quotient  
 $R^*(J)/\mathcal{I}_{K}(J)$. It is well-defined since $\mathcal{I}_K(J)$ is 
 closed under $\mathrm{d}$. Clearly by relations \eqref{eq:relation1} and 
 \eqref{eq:relation2}, each monomial in $R^*_K(J)$ can be uniquely written 
 in the square-free form $\wt{v}^{\sigma}\wt{u}^{\tau}=\prod_{i=1}^mx^i$, 
 where $x^i=\wt{v}^i$ if $i\in\sigma$, $x^{i}=\wt{u}^i$ if $i\in \tau$ and 
 $x^i=1$ otherwise,
 $\sigma,\tau$ disjoint subset of $[m]$, $\sigma\in K$. 
 
 First let $J=(\bm{1})$ be constant with integers $1$, and denote 
 $R^*_K(\bm{1})$ simply by $R^*_K$. A comparison
 of $R^*_K$ with \eqref{eq:cup2} shows that, if we identify
 $\wt{v}^i$ with $u^i$, $\wt{u}^i$ with $t^i$ and the identity $1$ with 
 the void word, then we have the cochain isomorphism
 $\left(R^*_{K},\cdot\right)\cong \left(C^{*}( (D^1,S^0)^K),\smile\right)$. 
 Together with Theorem \ref{thm:Whitney}, we see that
 \[\left(H^*(R^*_{K(J)}),\cdot\right)\cong 
   \left(H^*( (D^1,S^0)^{K(J)}),\smile\right).\] 
 
   Notice that $R^*_{K(J)}$ has $2d(J)$ generators, $d(J)=\sum_{i=1}^{m}j_i$,
   which we shall still denote by $u^i$ of degree $1$ and $t^i$ of 
   degree $0$, $i=1,2,\ldots, d(J)$.
 Actually,
 the algebra $H^*(R^*_{K(J)})$ can be refined as follows:
 \begin{thm}\label{thm:iso}
   We have isomorphisms
   \begin{equation}
	  H^*(R^*_K(J))\cong H^*(R^*_{K(J)})\cong 
	  H^*( (D^1,S^0)^{K(J)})	\cong H^*((D^{j_i},S^{j_i-1})_{i=1}^m)^K
	   \label{iso:J}
   \end{equation}
   of graded algebras. Moreover, by setting 
   \begin{equation}
	(\sigma,\omega)_J=\sum_{k\in\omega}(j_k-1)\mathrm{card}(\{i\in\sigma\mid i>k\})+
   \sum_{k\in\omega}(j_k-1)\sum_{r\in\omega,r>k}(j_r-1)
	 \label{def:soJ}
   \end{equation}
   for $\sigma\subset \omega\subset [m]$, the mapping 
   \begin{equation}
	 \begindc{\commdiag}[15]
   \obj(0,1)[aa]{$\eta_J\co \bigoplus_{\omega\subset [m]}\widetilde{C}^*(K_{\omega})$}
   \obj(7,1)[bb]{$ R^*_K(J)$}
   \obj(0,0)[cc]{$\sigma^*_{\omega}$}
   \obj(7,0)[dd]{$ (-1)^{(\sigma,\omega)_J}\wt{v}^{\sigma}\wt{u}^{\omega\setminus\sigma}$}
   \mor{aa}{bb}{}
   \mor{cc}{dd}{}[+1,6]
 \enddc
	 \label{def:etaJ}
   \end{equation}
   yields a cochain map ($\sigma^*_{\omega}\in \widetilde{C}^*(K_{\omega})$ is a dual simplex)
	 which induces an isomorphism 
	 \[\bigoplus_{\omega\subset [m]}\widetilde{H}^*(K_{\omega})\cong H^*(R^*_K(J))\]
   of ungraded $\mathbb{Z}$-modules.
 \end{thm}
 \begin{remark}\label{rem:iso}
   The second isomorphism of \eqref{iso:J} 
   follows from Lemma \ref{lem:BBCG}. 
   When $j_i$
   is even, $i=1,2,\ldots,m$, this can be deduced from 
   \cite[Proposition 6.2]{BBCG10b}.
   In particular, when $J=(\bm{2})$ is constant with integers $2$, 
   $R^{*}_K(\bm{2})$ is the well-known algebra given in 
   \cite{BBP04} and \cite{Pan08} for the cohomology of 
   the moment-angle complex $(D^2,S^1)^K$.
    \end{remark}
 In order to prove Theorem \ref{thm:iso}, we consider the simplicial 
 wedge $K(v_i)$ on the $i$-th vertex (see \eqref{def:sw}). 
  Denote by $\varpi_i\co R_K^*\to R_{K(v_i)}^*$ 
  the additive homomorphism such that for each monomial 
  $u^{\sigma}t^{\tau}\in R_K^p$,
  \begin{equation}
	\varpi_i(u^{\sigma}t^{\tau})=\begin{cases}
		u^{\chi_i(\sigma)}t^{\chi_i(\tau)}\in 
		R^p_{K(v_i)} 
	  &\text{if $i\not\in \sigma\cup\tau $},\\
	  u^{\chi_i(\sigma)}t^{\chi_i(\tau)}u^{i+1}\in 
	  R^{p+1}_{K(v_i)} &\text{otherwise},
	\end{cases}
	\label{def:varpi}
  \end{equation}
 where $\chi_i$ is a label-shifting map such that
 each $\{i_k\}_{k=0}^l\subset [m]$ is sent to $\{i'_k\}_{k=0}^l\subset [m+1]$, 
in which $i'_k=i_k$ if $i_k\leq i$ and $i'_k=i_k+1$ otherwise
(thus the label $i+1$ is skipped in the image); 
$\chi_i(\emptyset)=\emptyset$. It can be checked easily that $\varpi_i$
is well-defined and preserves the differential $\mathrm{d}$ on
both sides.

In what follows, let $R^*_K|_{\omega}$ be the subalgebra generated by 
$u^{\sigma}t^{\omega\setminus\sigma}$, $\sigma\subset\omega$. Clearly 
  as a $\mathbb{Z}$-module, with $\omega$ running through subsets of $[m]$, 
  $R^*_K$ is a direct sum with summands $R^*_{K}|_\omega$.
  \begin{lem}\label{lem:varpi}
  The mapping $\varpi_{i}\co R_K^*\to R_{K(v_i)}^* $  
  induces an additive isomorphism on passage to cohomology.
  More precisely,  
  $\varpi_i$ induces isomorphisms
  \begin{equation}
	 H^p(R^*_{K}|_\omega)\cong H^{p+1}(R^*_{K(v_i)}|_{\chi_i(\omega)\cup\{i+1\}})
	\label{iso:withi}
  \end{equation}
  if $i\in\omega$, and isomorphisms 
  \begin{equation}
	H^p(R^*_{K}|_\omega)\cong H^{p}(R^*_{K(v_i)}|_{\chi_i(\omega)})
	\label{iso:withouti}
  \end{equation}
  if $i\not\in\omega$, for each $p\geq 0$.
 \end{lem}
 \begin{proof}
   First observe that if $i\not\in\omega$, then 
   the full subcomplex $K(v_i)_{\chi_i(\omega)}$ and $K_{\omega}$ are 
   simplicially isomorphic. Then
   \eqref{iso:withouti} follows from Proposition \ref{prop:iso1}.
   It remains to prove \eqref{iso:withi}. Suppose $i\in\omega$, 
   and consider the short exact sequence
   \[\begin{CD}
	   0@>>> R^*_{K}|_{\omega}@>\varpi_i>>  R^{*}_{K(v_i)}|_{\chi_i(\omega)\cup\{i+1\}}
	   @>>>Q^*@>>>0,
	 \end{CD}\]
	 in which $Q^{*}$ is the graded quotient.	
	 Let $f\co \widetilde{C}^{*}(K(v_i)_{\chi_i(\omega)})\to Q^{*}$ 
be the additive homomorphism sending each dual $p$-simplex 
$\sigma^*_{\chi_i(\omega)}$ to 
the element with representative 
$u^{\sigma}t^{\chi_i(\omega)\setminus\sigma}t^{i+1}$. We see that 
$f$ is a cochain isomorphism shifting the degrees up by one, since 
\[\mathrm{d}(u^{\sigma}t^{\tau}t^{i+1})=\mathrm{d}(u^{\sigma}t^{\tau})t^{i+1}+
 (-1)^{p+1}u^{\sigma}t^{\tau}\mathrm{d}t^{i+1},\]
 where the second summand vanishes in $Q^{*}$ (see \eqref{def:varpi}). It
 is easy to see that $K(v_i)_{\chi_i(\omega)}$ is 
 the star of $v_i$ in $K(v_i)$, which is acyclic, so
 is
 $Q^*$. Then \eqref{iso:withi} follows.
  \end{proof}
Now we specify a sequence of simplicial wedge constructions for $K(J)$:
it can be checked that 
 \begin{align*}
	K(J)=&K\underbrace{(v_m)(v_{m+1})\cdots (v_{m+j_m-2})}_{j_m-1}
	(v_{m-1})(v_m)\cdots (v_{m-1+j_{m-1}-2})\cdots
   \numberthis \label{eq:KJ}
    \\
	&(v_{m-i+1+j_{m-i+1}-2})\underbrace{(v_{m-i})(v_{m-i+1})\cdots
	(v_{m-i+j_{m-i}-2})}_{j_{m-i}-1}
	\\& (v_{m-i-1})\cdots
	(v_{2+j_{2}-2})\underbrace{(v_1)(v_2)\cdots (v_{1+j_1-2})}_{j_1-1},
  \end{align*}
 where the block marked by $j_i-1$ is deleted if $j_i=1$.
 \begin{proof}[Proof of Theorem \ref{thm:iso}]
   With respect to \eqref{eq:KJ}, we define $m$ composite homomorphisms 
   \[\varpi_{j_k}=\varpi_{k+j_k-2}\varpi_{k+j_k-3}\ldots\varpi_{k+1}\varpi_{k},\]
   $k=1,2,\ldots,m$, and denote their composition by 
   \[ \varpi_J=\varpi_{j_1}\varpi_{j_2}\ldots \varpi_{j_m}.\]
   Suppose 
   $\omega=\{i_1,i_2,\ldots,i_l\}$ is a subset of $[m]$, $i_1<i_2<\ldots<i_l$.
   Let  $u^{\sigma}t^{\omega\setminus\sigma}$ ($\sigma\subset\omega$) 
   be the monomial $\prod_{k=1}^lx^{i_k}$ 
   with $x^{i_k}=u^{i_k}$ if $i_k\in \sigma$, and 
   $x^{i_k}=t^{i_k}$ otherwise. By definition \eqref{def:varpi} 
   and a straightforward calculation, we have
\begin{align}
 \varpi_J(u^{\sigma}t^{\omega\setminus\sigma})=
 \varpi_{j_1}\varpi_{j_2}\ldots\varpi_{j_{i_m}}(x^{i_1}\ldots x^{i_l})
 =x^{\widetilde{i_1}}x^{\wt{i_2}}\ldots x^{\wt{i_l}}x^{\wt{B}_{i_l}}
 x^{\wt{B}_{i_{l-1}}}\ldots x^{\wt{B}_{i_1}},
 \label{eq:hard}
\end{align}
where $\wt{k}=k+\sum_{r<k}(j_{r}-1)$ and 
\[
	\widetilde{B}_{k}=\{k+1+\sum_{r<k}(j_r-1),
  k+2+\sum_{r<k}(j_r-1),\ldots, k+j_k-1+\sum_{r<k}(j_r-1)\}, 
\]
which is empty if $j_k=1$.   Let 
$(\sigma,\omega)_J$ be the $\mathrm{mod}$ $2$ integer such that 
\[x^{\widetilde{i_1}}x^{\wt{i_2}}\ldots x^{\wt{i_l}}x^{\wt{B}_{i_l}}
  x^{\wt{B}_{i_{l-1}}}\ldots x^{\wt{B}_{i_1}}=
  (-1)^{(\sigma,\omega)_J}x^{\widetilde{i_1}}x^{\wt{B}_{i_1}}
  x^{\widetilde{i_2}}x^{\wt{B}_{i_2}}\ldots 
  x^{\widetilde{i_l}}x^{\wt{B}_{i_l}}.\]
  Clearly by \eqref{eq:relation1},  $(\sigma,\omega)_J$ coincides with \eqref{def:soJ}. 
  As a conclusion, 
 by setting $\wt{v}^k=u^{\wt{k}}u^{\wt{B}_k}$ and $\wt{u}^k=t^{\wt{k}}u^{\wt{B}_k}$, 
 $k=1,2,\ldots,m$, we see that $R^*_{K}(J)$ is embedded in $R^*_{K(J)}$ 
 as a subalgebra generated by the image of $\varpi_J\co R^*_{K}\to R^*_{K(J)}$. 
 Therefore, on passage to cohomology, 
 the embedding above induces an isomorphism, by Lemma \ref{lem:varpi}. 
 
 The second statement follows from the first one, together with
 Proposition \ref{prop:iso1}. 
 Here we check it directly. Note that by \eqref{def:soJ}, 
 for $i\in\omega\setminus\sigma$,
  we have
\[(\sigma\cup\{i\},\omega)_J-(\sigma,\omega)_J=\sum_{k\in\omega,k<i}(j_k-1).\]
 It turns out that 
 \begin{align*}
	\eta_J( \mathrm{d}\sigma^*_{\omega} )
	&=\sum_{\substack{i\in\omega\setminus\sigma\\
	  (\sigma\cup\{i\})\in K_{\omega}}}
	  (-1)^{\mathrm{card}\left( \{k\in\sigma\mid k<i\} \right)}
	  \eta_J\left( (\sigma\cup\{i\})^*_{\omega} \right)
	  \\
	 &=\sum_{\substack{i\in\omega\setminus\sigma\\
	   (\sigma\cup\{i\})\in K_{\omega}}}
	  (-1)^{\mathrm{card}\left( \{k\in\sigma\mid k<i\}\right)+(\sigma\cup\{i\},\omega)_J}
	  \widetilde{v}^{\sigma\cup\{i\}}\widetilde{u}^{\omega\setminus(\sigma\cup\{i\})}\\
	  &=\sum_{\substack{i\in\omega\setminus\sigma\\
		(\sigma\cup\{i\})\in K_{\omega}}}
	  (-1)^{\mathrm{card}\left( \{k\in\sigma\mid k<i\}\right)+
	  \sum_{k<i,k\in\omega}(j_k-1)+(\sigma,\omega)_J}
	  \widetilde{v}^{\sigma\cup\{i\}}\widetilde{u}^{\omega\setminus(\sigma\cup\{i\})}\\
      &=\sum_{\substack{i\in\omega\setminus\sigma\\
		(\sigma\cup\{i\})\in K_{\omega}}}
	  (-1)^{\sum_{k<i,k\in\sigma}j_k+\sum_{k<i,k\in\omega\setminus\sigma}(j_k-1)+(\sigma,\omega)_J}
	  \widetilde{v}^{\sigma\cup\{i\}}\widetilde{u}^{\omega\setminus(\sigma\cup\{i\})}\\
	  &=\mathrm{d}\left( (-1)^{(\sigma,\omega)_J}
	  \widetilde{v}^{\sigma}\widetilde{u}^{\omega\setminus\sigma} \right)
	  =\mathrm{d}\eta_J(\sigma^*_{\omega}),
	\end{align*}
	where we have used \eqref{eq:Leib} in the last two lines.
\end{proof}
   In particular, when $J$ is the constant tuple $(\bm{2})$, 
   the sign $(\sigma,\omega)_{(\bm{2})}$ was
   understood in Bosio and Meersseman \cite[Theorem 10.1]{BM06}, 
   using the intersection
   of submanifolds (see also \cite[Theorem 5.1]{Pan08}).
   \begin{exm}	\label{exm:GL13}
	 Let $K$ be the complex dual to the truncated cube in Figure \ref{fig:cube},
	 such that each vertex is labeled by the number in the center 
	 of the associated $2$-face.
	 In \cite[Theorem 3.1]{GL13}, it 
	 was proved that $H^*(R^*_K)$ and $H^*(R^*_K(\bm{2}))$ are \emph{not} 
	 isomorphic as ungraded rings, even with $\mathbb{Z}_2$ coefficients. 
	 Following their approach, here we illustrate the difference of the two
	 rings, by Theorem \ref{thm:iso}.
 	 \begin{figure}
   \begin{center}
           \includegraphics[width=7cm]{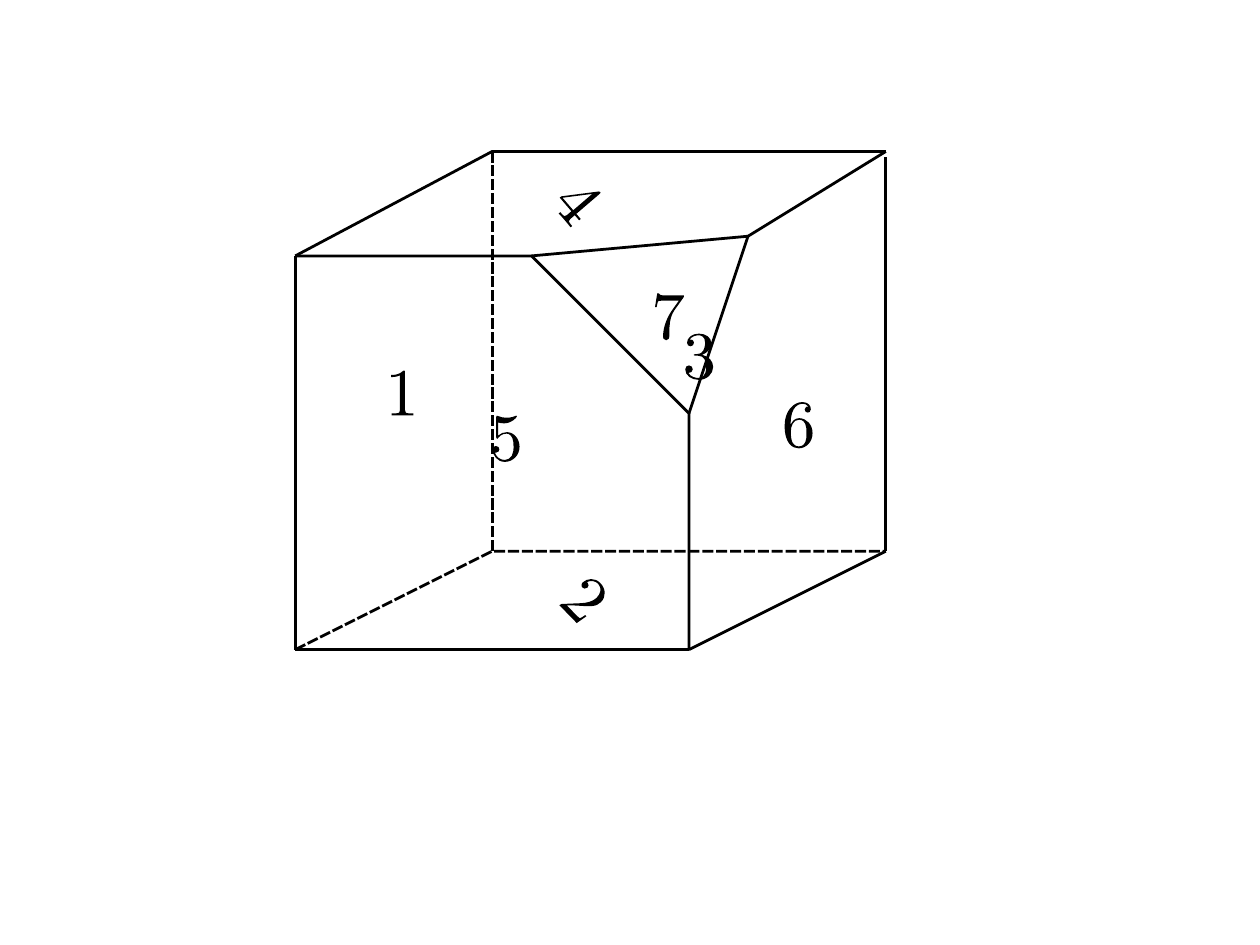}
   \end{center}
   \caption{The truncated cube}
   \label{fig:cube}
 \end{figure}
  First by \cite{GL13}, we have the homeomorphisms
  \[(D^1,S^0)^K\cong \left(S^1\times S^1\times S^1\right)\sharp
	\left(S^1\times S^1\times S^1\right)\sharp 7\left( S^1\times S^2 \right),\]
  and 
  \[ (D^2,S^1)^K\cong\partial\left(M^9_{-1}\times D^2\right)\sharp
	3\left( S^3\times S^7 \right)\sharp 3\left( S^4\times S^6 \right)\sharp
	\left( S^5\times S^5 \right)
\]
where $M_{-1}^9$ is the obtained from $S^3\times S^3\times S^3$ by 
subtracting an open ball.\footnote{$(D^1,S^0)^K$ and $(D^2,S^1)^K$ 
can be smoothed canonically to make the two homeomorphisms above into
diffeomorphisms; see \cite{GL13}.}

 It is easy to check that $\widetilde{H}^*(K_{\omega})$ is 
 non-trivial only when:
 \begin{enumerate}
   \item [(I)] $\omega=I_{1,2,3}\cup\{7\}$  
	 where $I_{1,2,3}$ can be any non-empty subsets of $\{1,2,3\}$ ($7$ cases);
	 $\omega=\{1,2,3,4,5,6\}\setminus I_{1,2,3}$ ($7$ cases);
   \item [(II)] $\omega=\{1,6,7\}\setminus I_7$, $\omega=\{2,4,7\}\setminus I_7$ and
	 $\omega=\{3,5,7\}\setminus I_7$, where $I_7=\emptyset$ or $I_7=\{7\}$ 
	 ($6$ cases); $\omega=\{1,3,5,6,7\}\setminus I_7$, 
	 $\omega=\{2,3,4,5,7\}\setminus I_7$ and $\omega=\{1,2,4,6,7\}\setminus I_7$ 
	 ($6$ cases).
 \end{enumerate} 
  First we consider $H^*(R^*_K)$.   With respect to (I), we have $7$ pairs of
  representatives $u^{7}t^{ I_{1,2,3}}$ of degree $1$ and 
  $u^{4,5}t^{\{1,2,3,6\}\setminus I_{1,2,3}}$ of degree $2$, respectively, 
  with $I_{1,2,3}$ running
  through non-empty subsets of $\{1,2,3\}$, such that 
  their product $-u^{4,5,7}t^{1,2,3,6}$ generates 
  $H^{3}(R^*_K)\cong H^3( (D^1,S^0)^K)$.  For (II), we choose representatives  
  $u^{1}t^{6,7}$, $u^{2}t^{4,7}$ and $u^{3}t^{5,7}$ of degree $1$, with their mutually 
  two-fold products $u^{1,2}t^{4,6,7}$, $u^{1,3}t^{5,6,7}$ and $u^{2,3}t^{4,5,7}$,
  respectively, together with the three-fold product $u^{1,2,3}t^{4,5,6,7}$.
  Analogously, we choose representatives 
  $u^{1}t^{6}(1-t^7)$, $u^2t^4(1-t^7)$ and 
  $u^{3}t^5(1-t^7)$, together with their mutually two-fold 
  and three-fold products. 
  It can be checked that, in this way we give a presentation of
  $H^*(R^*_K)$.   
  
  Next we turn to $H^*(R^*_{K}(\bm{2}))$. For (I), we can simply use
  the previous argument, replacing $u^i$ and $t^i$ by $\wt{v}^i$ and 
  $\wt{u}^i$, respectively. Notice that the degrees have changed 
  accordingly. For (II), however, we have to modify the representatives
  for products, since the relations $(\wt{u}^i)^2=\wt{u}^i$ do not 
  hold any longer. Therefore, any non-trivial three-fold product 
  must come from a partition of $[7]$ by three parts: two of them
  have cardinality $2$ and the remainder has cardinality $3$. For 
  instance, if we have chosen 
  $[\wt{v}^{1}\wt{u}^6]$ and $[\wt{v}^2\wt{u}^{4,7}]$ associated to 
  $\widetilde{H}^0(K_{1,6})$ and $\widetilde{H}^0(K_{2,4,7})$ 
  by $\eta_{\bm{2}}$ (see \eqref{def:etaJ}), respectively, 
  then the remainder has to be $[\wt{v}^3\wt{u}^{5}]$ associated to 
  $\widetilde{H}^0(K_{3,5})$. 
\end{exm}
\section{Proof of Theorem \ref{thm:Whitney}}\label{proof:Whitney}
We will use the same notations as in the statement of the theorem.

 Let $A^n$ be the $n$-th skeleton of $A$ 
 (we set $A^{-1}=\emptyset$), and 
let $\left(S_*(A),\partial_s\right)$ be the singular chain 
complex of $A$ with differential $\partial_s$.

\begin{proof}[Proof of the additive isomorphism]
The method here is well-known: 
it suffices to show that, 
there is a chain map 
$f_A\co \left(C_*(A),\partial\right)\to \left(S_*(A),\partial_s\right)$ such that 
the image of $C_n(A)$ generates $H_{n}(A^n,A^{n-1})$ in each dimension $n\geq 0$.
We construct $f_A$ by induction on dimension $n$. 
This is clear when $n=0$. Suppose $n>0$ and 
$f_A$ is well-defined for $k<n$, with the desired property. 
Choosing a cell $e_n$ of dimension $n$, we see
by hypothesis that the image of $\partial(e_n)$ under $f_A$ is carried by 
the boundary $\mathrm{Bd}(e_n)$, which is topologically 
an $(n-1)$--sphere. Again by hypothesis, $f_A(\partial(e_n))$ generates 
the fundamental class of $\mathrm{Bd}(e_n)$. 
Therefore, by the long exact sequence of the pair 
$\left(e_n,\mathrm{Bd}(e_n)\right)$, we 
can choose from $S_n(e_n)$ a representative of $H_n(e_n,\mathrm{Bd}(e_n))$, 
say $c_n$, such that $\partial_sc_n=f_A(\partial(e_n))$. Then we define $f_A(e_n)=c_n$,
and in this way we extend $f_A$ over all $n$-cells. Then a standard argument
of cellular homology shows that, $f_A$ and its dual induce the isomorphisms 
$H_*(C_*(A),\partial)\cong H_*(A)$ and $H^*(C^*(A),\mathrm{d})\cong H^*(A)$, 
respectively.
\end{proof}

On products, however, we have to work carefully
at the (co)chain level. 
We shall proceed with singular (co)chains, and then compare
it with simplicial ones. 

Let $\mathrm{Top}$ be the category of topological spaces, and 
let $\mathrm{Top}^m$ be the category of $m$-fold Cartesian product
spaces, with $m$-tuple of continuous maps as morphisms. Clearly 
$\mathrm{Top}^m$ is a subcategory of $\mathrm{Top}$. 
Denote by $\mathrm{C(Ab)}$  the 
category of chain complexes with chain maps as morphisms.

Let $S_*\co \mathrm{Top}\to \mathrm{C(Ab)}$ be the functor of 
singular chain complexes, assigning to a space $X$ the singular chain complex
$\left( S_*(X),\partial_s \right)$; let $S_*^m\co \mathrm{Top}^m\to \mathrm{C(Ab)}$
be the functor assigning to a product $X=\prod_{i=1}^m X_i$ the tensor product
$\left(\bigotimes_{i=1}^mS_*(X_i),\partial_{s'}\right)$ ($ \partial_{s'}$ follows 
\eqref{eq:bound1}, with suitable chains replaced).

Recall that for an object $X=\prod_{i=1}^m X_i$ from 
$\mathrm{Top}^m$, the associated \emph{Alexander-Whitney chain map} 
$T_X^m\co \left(S_*(X),\partial_s\right)\to 
\left(S_*^m(X),\partial_{s'}\right)$ is defined by 
sending each singular $p$-simplex $\sigma\co\Delta^p\to X$ 
(the vertex set of $\Delta^p$ is $\{0,1,\ldots,p\}$)
to the sum 
\begin{equation}
  \sum_{0=k_0\leq k_1\leq\ldots\leq k_m=p}\otimes_{i=1}^m\pi_i\sigma|_{[k_{i-1},k_{i}]},
  \label{def:AW}
\end{equation}
where $\pi_i$ is the projection onto the $i$-th component and 
$\sigma|_{[k_{i-1},k_{i}]}$ is the restriction of $\sigma$ to
the face spanned by $\{k_{i-1},k_{i-1}+1,\ldots, k_i\}$. We see that 
$T^m\co S_*|_{\mathrm{Top}^m}\to S_*^m$ is a 
natural transformation between the two functors. Moreover, by Eilenberg--Zilber
Theorem \cite{EZ53}, $T^m$ induces a natural chain-homotopy equivalence when
the space is specified. 

Let $\left(S^*(X),\delta_s\right)$ be the singular cochain complex 
of $X$. For $\mathbb{Z}$-modules $M_i$ with their duals
$M_i^*=\mathrm{Hom}(M_i,\mathbb{Z})$, $i=1,2,\ldots,m$, 
we denote by $\theta_m\co\bigotimes_{i=1}^mM_i^* 
\to\mathrm{Hom}\left( \bigotimes_{i=1}^m M_i,\mathbb{Z} \right)$ 
the \emph{evaluation map}, which is a homomorphism given by
\[          \theta_m(\otimes_{i=1}^m f_i)(\otimes_{i=1}^m c_{i})=
  \prod_{i=1}^mf_i(c_{i}),\]
  where $f_i\in M_i^*$ and $c_i\in M_i$. 
Consider the sequence 
  \begin{equation*}
	\begin{CD}
	S_*(X)@>d_*>>S_*(X\times X)@>T_{X\times X}^2>>S_*^2(X),
  \end{CD}
\end{equation*}
where $d_*$ is induced by the diagonal map $d\co X\to X\times X$. We refer to
the composition $\tau_X=T^2_{X\times X}d_*$  as the \emph{Alexander-Whitney diagonal
approximation}. By definition, we see that  
$\tau\co S_*\to S_*^2$ is a natural transformation between the two functors.

The cup product $\smile\co S^*(X)\otimes S^*(X)\to S^*(X)$
is defined as the composition $\tau_X^*\theta_2$, where
$\tau_X^*$ is the dual of $\tau_X$.

To formulate the cap product, we need another homomorphism 
\begin{equation}
\begindc{\commdiag}[15]
\obj(0,1)[aa]{$h\co M^*\otimes M\otimes M$}
\obj(9,1)[bb]{$M$}
\obj(0,0)[cc]{$(m^*, m_1\otimes m_2)$}
\obj(9,0)[dd]{$ m^*(m_2) m_1$,}
\mor{aa}{bb}{}
\mor{cc}{dd}{}[+1,6]
 \enddc		 
 \label{def:h}
 \end{equation}
 where $M$ is a $\mathbb{Z}$-module with its dual $M^*$.
It can be checked by a straightforward calculation that 
for $c^p\in S^{p}(X)$ and 
$c_{r}'\in S^2_r(X\times X)$, we have
\begin{equation}
  \partial_s \circ h(c^p,c_r')=(-1)^{r-p}h(\delta_s c^p,c_r')+h(c^p,\partial_{s'}c_r').
  \label{eq:h}
\end{equation}
The cap product $\frown\co S^*(X)\otimes S_*(X)\to S_*(X)$ 
is then given by  
\[c^p\frown c_r=h(c^p,\tau_X(c_r)).\]

Now we consider functors from $\mathrm{Top}^m$ to $\mathrm{C(Ab)}$. 
For a space $X=\prod_{i=1}^mX_i$, 
we have the diagram
  \begin{equation}
  \begindc{\commdiag}[22]
  \obj(0,4)[aa]{$S_*(X)$}
  \obj(8,4)[bb]{$S_*(X)\otimes S_*(X)$}
  \obj(0,0)[cc]{$S^m_*(X)$}
  \obj(8,0)[dd]{$\bigotimes_{i=1}^m\left(S_*(X_i)\otimes S_*(X_i)\right)$,}
  \obj(6,2)[ee]{$S^m_*(X)\otimes S^m_*(X)$}
  \mor{aa}{bb}{$\tau_X$}
  \mor{aa}{cc}{$T_{X}^m$}
  \mor{bb}{ee}{$T_{X}^m\otimes T_X^m$}[-1,0]
  \mor{cc}{dd}{$\otimes_{i=1}^m \tau_{X_i}$}
  \mor{dd}{ee}{$T_{Shuf}$}
\enddc\label{diag:key}
 \end{equation}
 in which the homomorphism $T_{Shuf}$  
 is given by the shuffling:  
  \begin{equation}
	T_{Shuf}\left( \otimes_{i=1}^m(c_{p_i}\otimes c_{q_i}) \right)=(-1)^{(\bm{p},\bm{q})}
  \left( \otimes_{i=1}^mc_{p_i} \right)\otimes \left( \otimes_{i=1}^mc_{q_i} \right),
  \label{def:T_S}
\end{equation}
where $\bm{p}=(p_i)_{i=1}^m$, $\bm{q}=(q_i)_{i=1}^m$, $c_{p_i}\in S_{p_i}(X_i)$ and 
$c_{q_i}\in S_{q_i}(X_i)$, $i=1,2,\ldots, m$; $(\bm{p},\bm{q})$ is given in  \eqref{def:pq}.
It can be checked directly that $T_{Shuf}$ is a chain map. We see that  
chain maps
\[\psi_X=(T_X^m\otimes T_X^m)\circ \tau_X\quad 
	\text{ and }\quad 
	\psi_X'=T_{Shuf}\circ\left(\otimes_{i=1}^m \tau_{X_i}\right)\circ T_X^m\]
can be treated as natural transformations 
$\mathrm{Top}^m\to \mathrm{C(Ab)}$ which are specified
to $X$.
Using Acyclic Model Theorem, with models 
as $2m$-product of simplices (see \cite[Section 6, Chapter 5, p.~252]{Spa66} 
for the details of the case
 $m=2$, and the proof for the general case is similar), 
 it turns out that Diagram \eqref{diag:key} commutes, 
 up to natural chain homotopy. This means that there is a natural 
 transformation $D\co\mathrm{Top}^m\to \mathrm{C(Ab)}$, 
 such that 
 \begin{equation}
   \psi_X-\psi_X'=D_X\partial_s+\partial_{s'}D_X
   \label{eq:psi}
 \end{equation}
 when $X$ is specified.
 
 \begin{lem}
   Let $c^{\bm{p}}=\otimes_{i=1}^mc^{p_i}\in \bigotimes_{i=1}^mS^{p_i}(X_i), 
   c^{\bm{q}}=\otimes_{i=1}^mc^{q_i}\in \bigotimes_{i=1}^mS^{q_i}(X_i)$ 
   be two cochains of degrees $|\bm{p}|$ and $|\bm{q}|$, 
   respectively.  Then 
   we have
   \begin{align*}
	& (T^m_X)^*\theta_m(c^{\bm{p}})\smile  (T^m_X)^*\theta_m(c^{\bm{q}})-
	 (-1)^{(\bm{p},\bm{q})}(T^m_X)^*
	 \theta_m\left( \otimes_{i=1}^m c^{p_i}\smile c^{q_i} \right) 
	 \numberthis\label{eq:cupf1}\\
	 =&
\delta_sB_1\left( c^{\bm{p}},c^{\bm{q}} \right)+
B_2\left( \delta_{s'}c^{\bm{p}},c^{\bm{q}} \right)+
B_3\left( c^{\bm{p}},\delta_{s'}c^{\bm{q}} \right)
   \end{align*}
where $(T^m_X)^*$ is the dual of 
$T^m_X$, $\delta_{s'}$ the coboundary operator of $\bigotimes_{i=1}^mS^*(X_i)$, 
and $B_i\co S_*^m(X)\otimes S_*^m(X)\to S_*(X)$ 
are three bilinear forms associated to $|\bm{p}|$ and $|\bm{q}|$, 
$i=1,2,3$. Additionally, given 
$c_{\bm{r}}=\otimes_{i=1}^m c_{r_i}$, $c_{r_i}\in S_{r_i}(X_i)$, 
we have
\begin{align*}
  &(T^m_X)^*\theta_m(c^{\bm{p}})\frown (T^{m}_X)^{-1}(c_{\bm{r}})
  -(-1)^{(\bm{r}-\bm{p},\bm{p})}(T_X^m)^{-1}\left( \otimes_{i=1}^mc^{p_i}\frown c_{r_i}
  \right)\numberthis \label{eq:capf1}\\
  =&\partial_s B_{1}'\left( c^{\bm{p}},c_{\bm{r}} \right)+
	 B_{2}'\left( \delta_{s'} c^{\bm{p}},c_{\bm{r}} \right)+
	 B_{3}'\left( c^{\bm{p}},\partial_{s'}c_{\bm{r}} \right)
\end{align*}
with three bilinear forms
$B_i'\co \bigotimes_{i=1}^mS^*(X_i)\otimes S_*^m(X)\to S_*^m(X)$ 
(with respect to $|\bm{p}|$, $|\bm{q}|$), $(T_{X}^m)^{-1}$  
the chain-homotopy inverse of $T^m_X$.
\end{lem}
 \begin{proof}
   This follows by a diagram chasing on Diagram \eqref{diag:key}. Let 
   $c^{\bm{p},\bm{q}}$ denote the cochain $c^{\bm{p}}\otimes c^{\bm{q}}$. 
   By assumption, we have
   \begin{align*}
	 \psi_X^*\left(\theta_{2m}(c^{\bm{p},\bm{q}})\right)&= \tau_X^* 
	 \theta_2\left( (T^m_X)^*\theta_m(c^{\bm{p}})
	 \otimes(T^m_X)^*\theta_m(c^{\bm{q}})\right) \\
	 &= 
	 (T^m_X)^*\theta_m(c^{\bm{p}})\smile  (T^m_X)^*\theta_m(c^{\bm{q}}),\\
   		 (\psi'_X)^*\left(\theta_{2m}(c^{\bm{p},\bm{q}})\right)&=
	 (T^m_X)^*\theta_m\left((-1)^{(\bm{p},\bm{q})}\otimes_{i=1}^m\tau_{X_i}^*
	 \theta_2(c^{p_i}\otimes c^{q_i})\right)\\ 	 
	 &=(-1)^{(\bm{p},\bm{q})}(T^m_X)^*
\theta_m\left( \otimes_{i=1}^m c^{p_i}\smile c^{q_i} \right).
   \end{align*}
   Then \eqref{eq:cupf1} follows from 
   a comparison of the two equations above, together with \eqref{eq:psi}.
   It remains to prove \eqref{eq:capf1}. It can be checked directly that
\begin{align*}
     (T^m_X)^*\theta_m(c^{\bm{p}})\frown (T^{m}_X)^{-1}(c_{\bm{r}})=& 
	 h\left( (T_X^m)^*\theta_m(c^{\bm{p}}),\tau_X(T_X^m)^{-1}(c_{\bm{r}}) \right)
	 \\
	 =&(T_X^m)^{-1}T_X^m
	 h\left( (T_X^m)^*\theta_m(c^{\bm{p}}),\tau_X(T_X^m)^{-1}(c_{\bm{r}}) \right)+
	 \\&\partial_s B_{11}'\left( c^{\bm{p}},c_{\bm{r}} \right)+
	 B_{12}'\left( \delta_{s'} c^{\bm{p}},c_{\bm{r}} \right)+
	 B_{13}'\left( c^{\bm{p}},\partial_{s'}c_{\bm{r}} \right),
  \end{align*}
  where bilinear forms $B_{1i}'$ comes from the chain-homotopy between
  $(T_X^m)^{-1}T_X^m$ and the identity of $S_*(X)$, $i=1,2,3$, together with
  \eqref{eq:h}.
  On the other hand, again by \eqref{eq:psi},
  we see that
  \begin{align*}
	&T_X^m
	 h\left( (T_X^m)^*\theta_m(c^{\bm{p}}),\tau_X(T_X^m)^{-1}(c_{\bm{r}}) \right)
	 =h\left( \theta_m(c^{\bm{p}}),
	 \left( T_X^m\otimes T_X^m \right)\tau_X(T_X^m)^{-1}(c_{\bm{r}}) \right)\\
	 =&h\left( \theta_m(c^{\bm{p}}),
	 T_{Shuf}\left( \otimes_{i=1}^m\tau_{X_i} \right)(c_{\bm{r}}) \right)+
	 \partial_s B_{21}'\left( c^{\bm{p}},c_{\bm{r}} \right)+
	 B_{22}'\left( \delta_{s'} c^{\bm{p}},c_{\bm{r}} \right)+\\
	 &B_{23}'\left( c^{\bm{p}},\partial_{s'}c_{\bm{r}} \right).
  \end{align*}
  By expanding $\left( \otimes_{i=1}^m\tau_{X_i} \right)(c_{\bm{r}})$ 
  with \eqref{def:AW}
  and then matching the degrees, we have
  \[h\left( \theta_m(c^{\bm{p}}),
	 T_{Shuf}\left( \otimes_{i=1}^m\tau_{X_i} \right)(c_{\bm{r}}) \right)=
	 (-1)^{(\bm{r}-\bm{p},\bm{p})}\otimes_{i=1}^mc^{p_i}\frown c_{r_i}.\]
	 Clearly \eqref{eq:capf1} follows from the equations above. 
 \end{proof}

Now let $\left( X,A \right)$ be the pair in the statement of Theorem
\ref{thm:Whitney}. Additionally, 
suppose that each $K_i$ is given a partial ordering on the vertex set, 
which induces a total ordering on each simplex. 
To complete the proof, 
it remains to compare the singular (co)chain complex with the simplicial one.

Recall that there is a chain-homotopy equivalence 
$\iota_i\co C_*(K_i)\to S_*(|K|)$
for each $i=1,2,\ldots,m$, such that $\iota_i$ sends
each $p$-simplex $[v_0,v_1,\ldots, v_p]$ to the singular
simplex linearly spanned by associated vertices, say $l(v_0,v_1,\ldots,v_p)$, 
with $v_0<v_1<\ldots<v_p$ in the given ordering; moreover, the chain-homotopy inverse 
$\iota_i^{-1}$ satisfies 
$\iota_i^{-1}(l(v_0,v_1,\ldots,v_p))=[v_0,v_1,\ldots,v_p]$ 
(see \cite[Theorem 34.3, pp.~194--195]{Mun84}). 
It can be checked that  diagrams
\begin{equation}
  \begin{CD}
	S^*(|K_i|)\otimes S^*(|K_i|)@>\smile>> S^*(|K_i|)\\
	@V V\iota^*_i\otimes\iota^*_iV  @V V\iota^*_iV\\
	C^*(K_i) \otimes C^*(K)@>\smile>> C^*(K_i)
  \end{CD}
\   \text{and} \
\begin{CD}
	S^*(|K_i|)\otimes S_*(|K_i|)@>\frown>> S_*(|K_i|)\\
	@V V\iota^*_i\otimes\iota^{-1}_iV  @V V\iota^{-1}_iV\\
	C^*(K_i) \otimes C_*(K_i)@>\frown>> C_*(K_i)
  \end{CD}
  \label{diag:cupcapf}
\end{equation}
commute, where the simplicial cup and cap products in the bottom rows 
are given by \eqref{def:cup0} and \eqref{def:cap0}, 
respectively. 

It follows that the map
$\iota\co \left(C_*(X),\partial\right)\to
\left(S_*^m(X),\partial_{s'}\right)$,  $\iota=\otimes_{i=1}^m\iota_i$,
is a chain-homotopy equivalence with its inverse 
$\iota^{-1}=\otimes \iota_i^{-1}$.\footnote{This 
  is well-known: suppose $D_i$ is 
  the chain-homotopy between $\alpha_i=\iota^{-1}_i\iota_i$ and the identity $\mathrm{id}_i$ 
  of $C_*(K_i)$, one can check that the homomorphism generated by
  \[\wt{D}(c_{\bm{r}})=\sum_{i=1}^m (-1)^{\sum_{k<i}r_k}\alpha_1(c_{r_1})
   \otimes\ldots\otimes\alpha_{i-1}(c_{r_{i-1}})\otimes D_i(c_{r_i})\otimes 
   \mathrm{id}_{i+1}(c_{r_{i+1}})\ldots\otimes\mathrm{id}_{m}(c_{r_m}),\]
   in which $c_{\bm{r}}=\otimes_{i=1}^mc_{r_i}$ ($c_{p_i}\in C_{r_i}(K_i)$), 
 gives the desired homotopy between $\iota$ and the identity of $C_*(X)$. 
 Similarly, we can define the chain-homotopy between 
 $\iota\iota^{-1}$ and the identity of $S_*^m(X)$.}
\begin{proof}[Proof of Theorem \ref{thm:Whitney}]
 Let $\upsilon_c\co C_*(A)\to 
 C_*(X)$ and 
 $\upsilon_{s}\co S_*(A)\to 
S_*(X)$ be the chain maps induced by 
 the inclusion $A\to X$, respectively, and let 
 $\upsilon^{*}_c$ and $\upsilon^*_s$ be their duals. 
 We see that $\upsilon^{*}_c$ sends
 a dual basis $\otimes_{i=1}^m\sigma^*_{p_i}$ to the one of 
 the same form, if $\prod_{i=1}^m|\sigma_{p_i}|\subset A$, and 
 to zero otherwise (when using the dual basis of the form
 above, we have used
 the evaluation map $\theta_m$ implicitly). By definition,
 $\upsilon_c^*$ preserves the cup products given 
 in \eqref{eq:cup}.

  Consider the composition 
  $(T^m)^{-1}\iota\upsilon_c\co C_*(A)\to S_*(X)$, 
  by the  naturality of $(T^m)^{-1}$, it gives
  rise to a chain map
  $C_*(A)\to S_*(A)$ 
  (since the image of $S_*^m(\prod_{i=1}^m|\sigma_{p_i}|)$ under 
  $(T^m)^{-1}$ lies in $S_*(\prod_{i=1}^m|\sigma_{p_i}|)$). 
  Thus we define \[f_A=(T^m)^{-1}\iota\upsilon_c;\] 
  it can be checked that $f_A$ satisfies the desired property in the proof 
 of the additive isomorphism $H_*(C_*(A),\partial)\cong H_*(A)$. 
 The argument above shows that the 
 chain-homotopy inverse of $f_A$, say $f_A^{-1}$, is then given
 by $\iota^{-1} T^m_X\upsilon_s\co S_*(A)
 \to C_*(A)$, whose dual
 \[(f_A^{-1})^*=\upsilon_s^*(T_X^m)^*(\iota^{-1})^*\co C^*(A) \to S^*(A)\]
 induces the isomorphism $H^*(C^*(A),\mathrm{d})\cong H^*(A) $. 
 Notice that $(f_A^{-1})^*$ can be defined over $C^*(X)$, 
 which we shall denote by $(\widetilde{f}_A^{-1})^*$: 
 we have 
 \begin{equation}
   (f_A^{-1})^*\upsilon_c^*=(\widetilde{f}_A^{-1})^*.
   \label{eq:pushout}
 \end{equation}
 Now we choose cochains $c^{\bm{p}}=\otimes_{i=1}^mc^{p_i}\in C^{|\bm{p}|}(X)$ 
 and $c^{\bm{q}}=\otimes_{i=1}^mc^{q_i}\in C^{|\bm{q}|}(X)$,
 where $c^{p_i}\in C^{p_i}(|K_i|)$ and $c^{q_i}\in C^{q_i}(|K_i|)$.
  By \eqref{eq:cupf1} and Diagram \eqref{diag:cupcapf}, we have 
 \begin{align*}
   & (f_A^{-1})^*\upsilon_c^*(c^{\bm{p}})\smile  
   (f_A^{-1})^*\upsilon_c^*(c^{\bm{q}})
   -(-1)^{(\bm{p},\bm{q})}
   (f_A^{-1})^*\upsilon_c^*\left( \otimes_{i=1}^m c^{p_i}
	 \smile c^{q_i} \right) \\=
	 & (\wt{f}_A^{-1})^*(c^{\bm{p}})\smile  
   (\wt{f}_A^{-1})^*(c^{\bm{q}})
   -(-1)^{(\bm{p},\bm{q})}
   (\wt{f}_A^{-1})^*\left( \otimes_{i=1}^m c^{p_i}
	 \smile c^{q_i} \right) \\=
	& \upsilon_s^*\delta_sB_1\left( (\iota^{-1})^*(c^{\bm{p}}),
	(\iota^{-1})^*(c^{\bm{q}}) \right)+
	\upsilon_s^*B_2\left( \delta_{s'}(\iota^{-1})^*(c^{\bm{p}}),
	(\iota^{-1})^*(c^{\bm{q}}) \right)+\\
&\upsilon_s^*B_3\left( (\iota^{-1})^*(c^{\bm{p}}),\delta_{s'}(\iota^{-1})^*(c^{\bm{q}}) \right);
 \end{align*}
 due to the naturality, the summation in the 
 last line above can be written as
 \begin{align*}
& \delta_s\upsilon_s^*B_1\left( (\iota^{-1})^*\upsilon_c^*(c^{\bm{p}}),
	(\iota^{-1})^*\upsilon_c^*(c^{\bm{q}}) \right)+
\upsilon_s^*B_2\left( (\iota^{-1})^*\upsilon_c^*(\mathrm{d}c^{\bm{p}}),
(\iota^{-1})^*\upsilon_c^*(c^{\bm{q}}) \right)+ \numberthis\label{eq:win1}\\
&\upsilon_s^*B_3\left( (\iota^{-1})^*\upsilon_c^*(c^{\bm{p}}),
(\iota^{-1})^*\upsilon_c^*(\mathrm{d}c^{\bm{q}}) \right).
 \end{align*}
 Therefore, we have proved part on cup products, since $\upsilon_c^*$ is 
 surjective, and \eqref{eq:win1} will vanish on passage to cohomology (notice
 that $\upsilon_c^*(c^*)\in C^{*}(A)$ is closed 
 means that $\upsilon_c^*(\mathrm{d}c^*)=0$, while $\mathrm{d}c^*$ 
 may not vanish in $C^*(X)$).

 Now we prove the part on cap products. Suppose
 $c_{\bm{r}}=\otimes_{i=1}^mc_{r_i}\in C_{|\bm{r}|}(X)$, $c_{r_i}\in C_{r_i}(|K_i|)$,
 such that $c_{\bm{r}}$ is, under $\upsilon_c$, 
 the image of a chain of the same form in $C_*(A)$, which we 
 shall denote as $\underline{c}_{\bm{r}}$.  Then by definition \eqref{def:cap0},
 we see that under $\upsilon_c$, 
 $\otimes_{i=1}^mc^{p_i}\frown c_{r_i}$ has a unique preimage 
 $\underline{\otimes_{i=1}^mc^{p_i}\frown c_{r_i}}\in C_*(A)$.
 Analogously to the previous case, by \eqref{eq:capf1} and the naturality, 
 we have
 \begin{align*}
   &(f_A^{-1})^*\upsilon_c^*(c^{\bm{p}})\frown f_A(\underline{c}_{\bm{r}})
   -(-1)^{(\bm{r}-\bm{p},\bm{p})}f_A\left( \underline{\otimes_{i=1}^mc^{p_i}\frown c_{r_i}}
  \right)\\
  =&(\wt{f}_A^{-1})^*(c^{\bm{p}})\frown (T^m_X)^{-1}\iota(c_{\bm{r}})
   -(-1)^{(\bm{r}-\bm{p},\bm{p})}(T^m_X)^{-1}\iota
   \left( \otimes_{i=1}^mc^{p_i}\frown c_{r_i}
  \right)\\
  =&\partial_s B_{1}'\left( (\iota^{-1})^*(c^{\bm{p}}),
  \iota(c_{\bm{r}}) \right)+
	 B_{2}'\left( \delta_{s'}(\iota^{-1})^*(c^{\bm{p}}),
	 \iota(c_{\bm{r}}) \right)+
	\\ 
	& B_{3}'\left( (\iota^{-1})^*(c^{\bm{p}}),
	\partial_{s'}\iota(c_{\bm{r}}) \right)\\
	=&\partial_s B_{1}'\left( (\iota^{-1})^*\upsilon_c^*(c^{\bm{p}}),\iota(c_{\bm{r}}) \right)+
	B_{2}'\left( (\iota^{-1})^*\upsilon_c^*(\mathrm{d}c^{\bm{p}}),
	 \iota(c_{\bm{r}}) \right)+
	\\ 
	& B_{3}'\left( (\iota^{-1})^*\upsilon_c^*(c^{\bm{p}}),
	\iota(\partial c_{\bm{r}}) \right),
\end{align*}
where the summation in the last line shall vanish on passage to (co)homology.
\end{proof}

%%% Please do NOT use ``\bysame'' command in your bibliography list.

\end{document}